\documentclass{article}

\title{Circularly compatible ones, $D$-circularity,
\\and proper circular-arc bigraphs}
\author{Mart\'in D.\ Safe\footnote{Departamento de Matem\'{a}tica, Universidad Nacional del Sur (UNS), Bah\'ia Blanca, Argentina and INMABB, Universidad Nacional del Sur (UNS)-CONICET, Bah\'ia Blanca, Argentina}}
\date{}

\usepackage{graphics}
\usepackage{fullpage}
\usepackage{amsmath}
\usepackage{amssymb}
\usepackage{amsthm}
\usepackage{floatrow}
\usepackage{subfig}

\usepackage[colorlinks,linkcolor=blue,citecolor=blue]{hyperref}

\usepackage{enumerate}

\newtheorem{thm}{Theorem}
\newtheorem{lem}[thm]{Lemma}
\newtheorem{cor}[thm]{Corollary}

\newcommand\FDCircR{\mathcal F_\mathrm{Dcirc}}
\newcommand\FDIntR{\mathcal F_\mathrm{Dint}}
\newcommand\FCCO{\mathcal F_\mathrm{CCO}}
\newcommand\ForbPCAB{\mathcal F_\mathrm{PCAB}}
\newcommand\ForbPIB{\mathcal F_\mathrm{PIB}}
\newcommand\MI[1]{M_{\mathit{I}}(#1)}
\newcommand\MIV{M_{\mathit{IV}}}
\newcommand\MV{M_{\mathit{V}}}
\newcommand\miop{\mathbin{\oplus}}
\newcommand\MIast[1]{M_{\mathit{I}}^*(#1)}

\newcommand\MVast{M_{\mathit{V}}^*}

\newcommand\trans{^{\mathit t}}
\newcommand\ForbRow{\mathcal F_{\textup{circR}}}
\newcommand\size{\mathop{\textrm{size}}}
\newcommand\id{\mathop{\textup{id}}}
\newcommand\leqR{\mathbin{\preccurlyeq_{\mathrm r}}}
\newcommand\leqC{\mathbin{\preccurlyeq_{\mathrm c}}}
\newcommand\leqCprime{\mathbin{\preccurlyeq'_{\mathrm c}}}
\newcommand\leqCast{\mathbin{\preccurlyeq^\ast_{\mathrm c}}}
\newcommand\leqCC{\mathbin{\preccurlyeq^+_{\mathrm c}}}
\newcommand\lessR{\mathbin{\prec_{\mathrm r}}}
\newcommand\lessC{\mathbin{\prec_{\mathrm c}}}
\newcommand\leqX{\mathbin{\preccurlyeq}}
\newcommand\lessX{\mathbin{\prec}}
\newcommand\leqXX{\mathbin{\preccurlyeq^+}}
\newcommand\lessXX{\mathbin{\prec^+}}
\newcommand\XX{X^+}
\newcommand\intXX[1]{[#1]^+_{\leqX}}

\newcommand\ZetaUnoAst{\ensuremath{Z_1^*}}
\newcommand\ZetaDosAst{\ensuremath{Z_2^*}}
\newcommand\ZetaTresAst{\ensuremath{Z_3^*}}
\newcommand\ZetaCuatroAst{\ensuremath{Z_4^*}}
\newcommand\ZetaCinco{\ensuremath{Z_5}}
\newcommand\ZetaCincoTrans{\ensuremath{Z_5\trans}}
\newcommand\ZetaSeis{\ensuremath{Z_6}}
\newcommand\ZetaSiete{\ensuremath{Z_7}}
\newcommand\ZetaOcho{\ensuremath{Z_8}}
\newcommand\CoZetaUnoAst{\ensuremath{\overline{Z_1^*}}}
\newcommand\CoZetaDosAst{\ensuremath{\overline{Z_2^*}}}
\newcommand\CoZetaCuatroAst{\ensuremath{\overline{Z_4^*}}}
\newcommand\CoZetaSeis{\ensuremath{\overline{Z_6}}}

\begin{document}

\maketitle

\abstract{In 1969, Alan Tucker characterized proper circular-arc graphs as those graphs whose augmented adjacency matrices have the circularly compatible ones property. Moreover, he also found a polynomial-time algorithm for deciding whether any given augmented adjacency matrix has the circularly compatible ones property. These results allowed him to devise the first polynomial-time recognition algorithm for proper circular-arc graphs. However, as Tucker himself remarks, he did not solve the problems of finding a structure theorem and an efficient recognition algorithm for the circularly compatible ones property in arbitrary matrices (i.e., not restricted to augmented adjacency matrices only). In this work, we solve these problems. More precisely, we give a minimal forbidden submatrix characterization for the circularly compatible ones property in arbitrary matrices and a linear-time recognition algorithm for the same property. We derive these results from analogous ones for the related $D$-circular property. Interestingly, these results lead to a minimal forbidden induced subgraph characterization and a linear-time recognition algorithm for proper circular-arc bigraphs, solving a problem first posed by Basu, Das, Ghosh, and Sen~[\emph{J.\ Graph Theory}, 73(4):361--376, 2013]. Our findings generalize some known results about $D$-interval hypergraphs and proper interval bigraphs.}

\section{Introduction}\label{sec:intro}

In 1969, Alan Tucker~\cite{TuckerPhD} introduced the linearly compatible ones property in connection with a characterization of proper interval graphs in terms of their augmented adjacency matrices due to Fred Roberts~\cite{MR0281647}. In order to state this characterization, we now give the necessary definitions.

Let $\leqX$ be a linear order on some set $X$. If $a\leqX b$, then the \emph{linear interval of $\leqX$ with left endpoint $a$ and right endpoint $b$}, denoted $[a,b]_{\leqX}$, is the set $\{x\in X\colon\,a\leqX x\leqX b\}$. A \emph{linear interval of $\leqX$} is either the empty set or $[a,b]_{\leqX}$ for some $a,b\in X$ such that $a\leqX b$. A sequence $a_1a_2\ldots a_k$ is \emph{monotone on $X$} if $a_1,a_2,\ldots,a_k\in X$ and $a_1\leqX a_2\leqX\cdots\leqX a_k$. 

All matrices in this work are binary; i.e., have only $0$ and $1$ entries. We will usually identify each row $r$ of a matrix $M$ with the set of columns of $M$ having a $1$ at row $r$. For instance, we say a row $r$ is \emph{empty} if it has no $1$ entries, while a row $r$ is \emph{contained} in a row $s$, denoted $r\subseteq s$, if $s$ has a $1$ at each column where $r$ has a $1$. A row $r$ of $M$ is \emph{trivial} if it is either empty or the set of all columns of $M$. We adopt analogous conventions for the columns.

A \emph{biorder of a matrix $M$} is an ordered pair $(\leqR,\leqC)$ such that $\leqR$ and $\leqC$ are linear orders of the rows and of the columns of $M$,  respectively. A matrix $M$ has the \emph{linearly compatible ones property}~\cite[p.\ 43]{TuckerPhD} if $M$ admits some biorder $(\leqR,\leqC)$ such that: (i) each row of $M$ is a linear interval of $\leqC$; (ii) each column of $M$ is a linear interval of $\leqR$; and (iii) if $r_1,r_2,\ldots,r_p$ are all the nontrivial rows of $M$ in ascending order of $\leqR$ and $r_i$ equals the linear interval $[d_i,e_i]_{\leqC}$ for each $i\in\{1,2,\ldots,p\}$, then the sequences $d_1d_2\ldots d_p$ and $e_1e_2\ldots e_p$ are monotone on $\leqC$.\footnote{The definitions of the linearly compatible ones property and the circularly compatible ones property used in this work are taken directly from Tucker's PhD thesis~\cite{TuckerPhD} and should not be confused with the ones given in \cite{MR0276129,MR0309810,MR0295938} in the setting of symmetric matrices.} If so, $(\leqR,\leqC)$ is called a \emph{linearly compatible ones biorder of $M$}.

Proper interval graphs is a well-known class of intersection graphs. The \emph{intersection graph} of a family of sets is a graph having one vertex for each set of the family and having an edge joining two different vertices if and only if the sets of the family corresponding to these two vertices have nonempty intersection. A \emph{proper interval graph}~\cite{MR0252267} is the intersection graph of a family of intervals on a line no two of which are one a proper subset of the other. Proper interval graphs admit many different characterizations~\cite{MR1100040,MR2364171,MR1180196,MR1203643,RobertsPhD,MR0252267,MR0281647,MR666799,Wegner}. The result below is the aforementioned characterization by Roberts of proper interval graphs in terms of their augmented adjacency matrices. An \emph{augmented adjacency matrix} of a graph is any matrix that arises from an adjacency matrix by adding $1$'s all along the main diagonal.

\begin{thm}[\cite{MR0281647}]\label{thm:pig_lco} A graph is a proper interval graph if and only if its augmented adjacency matrix has the linearly compatible ones property.\end{thm}

Many notions which turn out to be equivalent to the linearly compatible ones property were subsequently introduced by Moore~\cite{MR0437403} ($D$-interval hypergraphs), Spinrad, Brandst{\"a}dt, and Stewart~\cite{MR917130} (adjacency and enclosure property), Sen and Sanyal~\cite{MR1271987} (monotone consecutive arrangements), and Lai and Wei~\cite{MR1440521} (forward-convex labelings). For instance, a matrix $M$ has the \emph{$D$-interval property}~\cite{MR0437403} if there is a linear order $\leqC$ of its columns such that each row of $M$ is a linear interval of $\leqC$ and the set difference $s-r$ is also a linear interval of $\leqC$ for any two rows $r$ and $s$ of $M$. Moore~\cite{MR0437403} characterized the $D$-interval property by minimal forbidden induced submatrices. Spinrad, Brandst{\"a}dt, and Stewart~\cite{MR917130} gave the first linear-time recognition algorithm for the $D$-interval property. A significantly simpler linear-time recognition algorithm was later proposed by Sprague~\cite{MR1369371}. More recently, Hell and Huang~\cite{MR2134416} proposed a linear-time recognition algorithm which, in addition, outputs a minimal forbidden submatrix of the input matrix $M$ whenever $M$ does not have the property.

Interestingly, the linearly compatible ones property also characterizes proper interval bigraphs. The \emph{bipartite intersection graph}~\cite{MR687568} of two families of sets $\mathcal F_1$ and $\mathcal F_2$ is a graph having a vertex for each element of $\mathcal F_1$ and for each element of $\mathcal F_2$ and such that a vertex corresponding to an element in $\mathcal F_1$ is adjacent to a vertex corresponding to an element in $\mathcal F_2$ if and only if these two elements have nonempty intersection. A \emph{proper interval bigraph}~\cite{MR1271987} is the bipartite intersection graph of two families $\mathcal F_1$ and $\mathcal F_2$ of intervals on a line where neither $\mathcal F_1$ nor $\mathcal F_2$ contains two intervals such that one is a proper subset of the other. If so, $\{\mathcal F_1,\mathcal F_2\}$ is called a \emph{proper interval bimodel} of the bipartite intersection graph of $\mathcal F_1$ and $\mathcal F_2$. Proper interval bigraphs admit several different characterizations~\cite{MR944699,MR2762475,MR3410848,MR3336533,MR2071482,MR2134416,MR1405855,MR1604968,MR1368750,MR1271987,MR1483762}. Interestingly, proper interval bigraphs are known to coincide with many other graph classes, including unit interval bigraphs~\cite{MR1271987}, bipartite permutation graphs~\cite{MR1368750}, bipartite asteroidal-triple-free graphs~\cite{MR0221974}, bipartite co-comparability graphs~\cite{MR0221974}, bipartite tolerance graphs~\cite{MR944699}, and the complement of two-clique circular-arc graphs~\cite{MR2071482}. The \emph{bipartite graph associated with a matrix $M$} has one vertex for each row and for each column of $M$ and the vertex corresponding to row $i$ is adjacent to the vertex corresponding to column $j$ if and only if the entry $(i,j)$ of $M$ is $1$. The result below follows by combining results from the works of Sen and Sanyal~\cite{MR1271987}, Lai and Wei~\cite{MR1440521}, and Moore~\cite{MR0437403} (see Subsection~\ref{ssec:LCO}).

\begin{thm}[\cite{MR1440521,MR0437403,MR1271987}]\label{thm:pib_cco} Proper interval bigraphs are precisely the bipartite graphs associated with matrices having the linearly compatible ones property.\end{thm}

The above theorem allows for translation back and forth between results about the linearly compatible ones property and results about proper interval bigraphs. In fact, the aforementioned linear-time recognition algorithms in~\cite{MR2134416,MR917130,MR1369371} for the linearly compatible ones property were originally formulated as recognition algorithms for proper interval bigraphs (or, equivalently, bipartite permutation graphs).

Tucker~\cite{TuckerPhD} introduced the circularly compatible ones property in order to characterize proper circular-arc graphs. If $\leqX$ is a linear order on some set $X$ and $a,b\in X$, then the \emph{circular interval of $\leqX$ with left endpoint $a$ and right endpoint $b$}, denoted $[a,b]_{\leqX}$, is either $\{x\in X\colon\;a\leqX x\leqX b\}$ if $a\leqX b$, or $\{x\in X\colon\,x\leqX b\text{ or }a\leqX x\}$ if $b\lessX a$. A \emph{circular interval of $\leqX$} is either the empty set or $[a,b]_{\leqX}$ for some $a,b\in X$. A sequence $a_1a_2\ldots a_k$ is \emph{circularly monotone on $\leqX$} if $a_1,a_2,\ldots,a_k\in X$ and $a_i\leqX a_{i+1}$ holds for all but at most one $i\in\{1,2,\ldots,k\}$ (where $a_{k+1}$ stands for $a_1$). A matrix $M$ has the \emph{circularly compatible ones property}~\cite[p.\ 30]{TuckerPhD} if $M$ admits some biorder $(\leqR,\leqC)$ such that: (i) each row of $M$ is a circular interval of $\leqC$; (ii) each column of $M$ is a circular interval of $\leqR$; and (iii) if $r_1,r_2,\ldots,r_p$ are all the nontrivial rows of $M$ in ascending order of $\leqR$ and $r_i=[d_i,e_i]_{\leqC}$ for each $i\in\{1,2,\ldots,p\}$, then the sequences $d_1d_2\ldots d_p$ and $e_1e_2\ldots e_p$ are circularly monotone on $\leqC$. If so, $(\leqR,\leqC)$ is called a \emph{circularly compatible ones biorder}. A \emph{proper circular-arc graph}~\cite{TuckerPhD,MR0309810} is the intersection graph of a family of arcs on a circle such that no two arcs are one a proper subset of the other. Proper circular-arc graphs admit several different characterizations~\cite{MR1081957,MR666799,TuckerPhD,MR0309810,MR0379298}. The result below shows that Theorem~\ref{thm:pig_lco} extends to proper circular-arc graphs and, in fact, this was the motivation behind the introduction of the circularly compatible ones property by Tucker.

\begin{thm}[{\cite[p.\ 36]{TuckerPhD}}]\label{thm:pca_cco} A graph is a proper circular-arc graph if and only if its augmented adjacency matrix has the circularly compatible ones property.\end{thm}

Based on the above characterization, Tucker was able to devise the first polynomial-time recognition algorithm for proper circular-arc graphs by reducing the problem to that of deciding whether any given augmented adjacency matrix has the circularly compatible ones property~{\cite[Section~2.2]{TuckerPhD}}. However, as Tucker himself remarks~\cite[p.\ 46]{TuckerPhD}, he did not solve the problems of finding a structure theorem and an efficient recognition algorithm for the circularly compatible ones property in arbitrary matrices (i.e., not restricted to augmented adjacency matrices only). We solve these problems. In order to do so, we study the following circular variant of the $D$-interval property. A matrix has the \emph{$D$-circular property} if there is a linear order $\leqC$ of its columns such that each row of $M$ is a circular interval of $\leqC$ and the set difference $s-r$ is also a circular interval of $\leqC$ for any two rows $r$ and $s$ of $M$. If so, $\leqC$ is a \emph{$D$-circular order} of $M$. Hypergraphs whose incidence matrices have the $D$-circular property were studied by K{\"o}bler, Kuhnert, and Verbitsky~\cite{MR3573905} (where they are called tight circular-arc hypergraphs).

Proper circular-arc bigraphs are defined analogously to proper interval bigraphs as follows. A \emph{proper circular-arc bigraph}~\cite{MR2071482} is the bipartite intersection graph of two families $\mathcal F_1$ and $\mathcal F_2$ of arcs on a circle where neither $\mathcal F_1$ nor $\mathcal F_2$ contains two arcs such that one is a proper subset of the other. If so, $\{\mathcal F_1,\mathcal F_2\}$ is called a \emph{proper circular-arc bimodel} of the bipartite intersection graph of $\mathcal F_1$ and $\mathcal F_2$. In~\cite{MR3336533}, proper circular-arc bigraphs were characterized in terms of a pair of linear orders of their vertices. Basu et al.~\cite{MR3065109} proved an analogue of Theorem~\ref{thm:pib_cco} for proper circular-arc bigraphs having a biadjacency matrix with no trivial rows, where the linearly compatible ones property is replaced with the $D$-circular property. Combining their result with our findings about the circularly compatible ones property, we derive the following analogue of Theorem~\ref{thm:pib_cco} for arbitrary proper circular-arc bigraphs by replacing the linearly compatible ones property with the circularly compatible ones property.

\begin{thm}\label{thm:pcab_cco} Proper circular-arc bigraphs are the bipartite graphs associated with matrices having the circularly compatible ones property.\end{thm}

Basu et al.~\cite{MR3065109} asked for an efficient recognition algorithm for proper circular-arc bigraphs; more recently, Das and Chakraborty~\cite{MR3662958} raised the same problem. We solve this problem by giving a linear-time algorithm for recognizing proper circular-arc bigraphs. Moreover, as a consequence of the above theorem and our minimal forbidden submatrix characterization of the circularly compatible ones property, we derive a minimal forbidden induced subgraph characterization for proper circular-arc bigraphs.

The main results of this work are a minimal forbidden submatrix characterization and a linear-time recognition algorithm for the $D$-circular property for arbitrary matrices. Moreover, we show that an arbitrary matrix has the circularly compatible ones property if and only if both the matrix and its transpose have the $D$-circular property. As a consequence, we derive a minimal forbidden submatrix characterization for the circularly compatible ones property together with a linear-time recognition algorithm (thus solving the aforementioned problems by Tucker~\cite{TuckerPhD}). Given the connection between proper circular-arc bigraphs and the circularly compatible ones property (Theorem~\ref{thm:pcab_cco}), these results lead to a minimal forbidden induced subgraph characterization and a linear-time recognition algorithm for proper circular-arc bigraphs (thus solving the problem first posed by Basu et al.~\cite{MR3065109}). Our recognition algorithms for matrix properties output either the linear order(s) required by the definition of the property or a minimal forbidden submatrix. Similarly, our recognition algorithms for graph classes either produce a bimodel as required by the definition of the class or a minimal forbidden induced subgraph.

This work is organized as follows. In Section~\ref{sec:defs}, we give basic definitions and notation and state some previous results about the consecutive-ones and the circular-ones properties. In Section~\ref{sec:Dcirc}, we give a minimal forbidden submatrix characterization and a linear-time recognition algorithm for the $D$-circular property and discuss their connection with some known results about the $D$-interval property. In Section~\ref{sec:LCO+CCO}, we argue that the linearly compatible ones property is equivalent to the $D$-circular property and give a minimal forbidden submatrix characterization and a linear-time recognition algorithm for the circularly compatible ones property. In Section~\ref{sec:PCAB+PIB}, we derive a minimal forbidden induced subgraph characterization and a linear-time recognition algorithm for proper circular-arc bigraphs. Some of the proofs of the more technical results are given in Appendix~\ref{sec:app}.

\section{Definitions and preliminaries}\label{sec:defs}

For each positive integer $k$, we denote by $[k]$ the set $\{1,2,\ldots,k\}$; if $k=0$, $[k]$ denotes the empty set. We also denote by $[k]$ the set $[k]$ endowed with the natural order. In the same vein, if $i,j\in[k]$, we write $[i,j]_{[k]}$ to denote the circular interval with left endpoint $i$ and right endpoint $j$ with respect to the natural order of $[k]$. By $\id_k$ we denote the identity function with domain $[k]$. Let $\leqX$ be a linear order on some finite set $X$. If $x,y\in X$, we write $x\lessX y$ to mean $x\leqX y$ and $x\neq y$; this convention also applies to linear orders denoted by $\leqX$ with some subscript and/or superscript (e.g., $\leqR$, $\leqC$, $\leqXX$, etc.). A set $S$ is \emph{properly contained in} or \emph{is a proper subset of} a set $T$ if $S$ is a subset of $T$ and $S\neq T$. Two sets $S$ and $T$ are \emph{incomparable} if none of them is a subset of the other.

\subsection*{Sequences}

Let $a=a_1a_2\ldots a_k$ be a sequence of length $k$. We call the \emph{shift of $a$} to the sequence $a_2a_3\ldots a_ka_1$ and the \emph{reversal of $a$} to the sequence $a_ka_{k-1}\ldots a_1$. We denote the length of any sequence $a$ by $\vert a\vert$. If $b$ is a sequence and $i\in[k]$, we say that \emph{$b$ occurs circularly in $a$ at position $i$} if $\vert b\vert\leq k$ and $a_ia_{i+1}\ldots a_{i+\vert b\vert-1}=b$ where subindices are modulo $k$. If $i\leq\vert a\vert-\vert b\vert+1$, we may simply say that $b$ \emph{occurs in $a$ at position $i$}.

If $a=a_1a_2\ldots a_k$ is \emph{binary} (i.e., each $a_i$ is either $0$ or $1$), we define the \emph{complement of $a$}, denoted by $\overline a$, as the sequence that arises from $a$ by interchanging $0$'s with $1$'s. A \emph{binary bracelet}~\cite{MR1857399} is a lexicographically smallest element in an equivalence class of binary sequences under shifts and reversals.

A sequence $\lambda$ is \emph{senary} if each of its element is $0$, $1$, $2$, $3$, $4$, or $5$. We define the \emph{complement} of a senary sequence $\lambda$, denoted by $\overline\lambda$, as the sequence that arises from $\lambda$ by interchanging $0$'s with $1$'s, $2$'s with $3$'s, and $4$'s with $5$'s.

\subsection*{Matrices}

Let $M$ and $M'$ be matrices. We say that \emph{$M$ contains $M'$ as a configuration} if some submatrix of $M$ equals $M'$ up to permutations of rows and of columns. We say that $M$ and $M'$ \emph{represent the same configuration} if $M$ and $M'$ are equal up to permutations of rows and of columns; otherwise, we say that $M$ and $M'$ \emph{represent different configurations}. We denote by $M^*$ the matrix that arises from $M$ by adding one last column consisting entirely of $0$'s. We denote the \emph{transpose} of $M$ by $M\trans$.

Let $M$ be a $k\times\ell$ matrix. We assume that the rows and columns of $M$ are labeled from $1$ to $k$ and from $1$ to $\ell$, respectively, as usual. By \emph{complementing row $i$ of $M$} we mean replacing, in row $i$, all $0$ entries by $1$'s and all $1$ entries by $0$'s. The \emph{complement of $M$}, denoted $\overline M$, is the matrix arising from $M$ by replacing all $0$ entries by $1$'s and all $1$ entries by $0$'s. If $a$ is a binary sequence of length $k$, we denote by $a\miop M$ the matrix that arises from $M$ by complementing those rows $i\in\{1,\ldots k\}$ such that $a_i=1$. A \emph{row map of $M$} is an injective function $\rho:[k']\to[k]$ for some positive integer $k'$. A \emph{column map of $M$} is an injective function $\sigma:[\ell']\to[\ell]$ for some positive integer $\ell'$. If $\rho:[k']\to[k]$ is a row map of $M$ and $\sigma:[\ell']\to[\ell]$ is a column map of $M$, we denote by $M_{\rho,\sigma}$ the $k'\times\ell'$ matrix such that, for each $(i,j)\in [k']\times[\ell']$, its $(i,j)$-entry is the $(\rho(i),\sigma(j))$-entry of $M$. We also write $M_\rho$ to mean $M_{\rho,\id_\ell}$. Notice that $M$ contains $M'$ as a configuration if and only if there is a row map $\rho$ and column map $\sigma$ of $M$ such that $M_{\rho,\sigma}=M'$. If, in addition, $a$ is a binary sequence whose length equals the number of rows of $M$, then $(a\miop M)_{\rho,\sigma}=a_\rho\miop M_{\rho,\sigma}$. It is also clear that, if $\rho$ and $\sigma$ are a row map and a column map of $M$ and $\rho'$ and $\sigma'$ are a row map and a column map of $M_{\rho,\sigma}$, then $\rho'\circ\rho$ and $\sigma'\circ\sigma$ are a row map and a column map of $M$, and $M_{\rho'\circ\rho,\sigma'\circ\sigma}=((M_{\rho,\sigma})_{\rho',\sigma'}$. If $s$ is a positive integer and $n_1,n_2,\ldots,n_s$ are pairwise different positive integers, we denote by $\langle n_1,n_2,\ldots,n_s\rangle$ the injective function with domain $[s]$ that transforms $i$ into $n_i$ for each $i\in[s]$. Hence, if $i\in[k]$, then $M_{\langle i\rangle}$ denotes the $1\times\ell$ matrix whose only row equals row $i$ of $M$.

The \emph{canonical order of the rows of $M$} is the linear order of the rows of $M$ as they occur from top to bottom. Similarly, the \emph{canonical order of the columns of $M$} is the linear order of the columns of $M$ as they occur from left to right. The \emph{canonical biorder of $M$} is the biorder $(\leqR,\leqC)$ where $\leqR$ and $\leqC$ are the canonical orders of the rows and of the columns of $M$, respectively.

\subsection*{Graphs}

All graphs in this work are simple; i.e., finite, undirected, and with no loops and no multiple edges. If $G$ is a graph and $X$ is some subset of its vertex set, the subgraph of $G$ \emph{induced by} $X$ is the graph having $X$ as vertex set and whose edges are the edges of $G$ having both endpoints in $X$. An \emph{induced subgraph} of some graph $G$ is the subgraph of $G$ induced by some subset of its vertex set. An \emph{isolated vertex} of a graph is a vertex of the graph adjacent to no vertex in the graph.

A \emph{stable set} of a graph is a set of pairwise nonadjacent vertices. A \emph{bipartition of a graph $G$} is a partition $\{X,Y\}$ of its vertex set into two (possibly empty) stable sets. A graph is \emph{bipartite} if it admits a bipartition. Let $G$ be a bipartite graph and let $\{X,Y\}$ be a bipartition of $G$. A \emph{biadjacency matrix of $G$ with respect to $X$ and $Y$} has one row for each vertex in $X$ and one column for each vertex in $Y$ and, for each $x\in X$ and each $y\in Y$, the entry in the intersection of the row corresponding to $x$ and the column corresponding to $y$ is $1$ if and only if $xy$ is an edge of $G$. The \emph{bipartite complement of $G$ with respect to $\{X,Y\}$} is the bipartite graph $G'$ with bipartition $\{X,Y\}$ such that, for each $x\in X$ and each $y\in Y$, $x$ is adjacent to $y$ in $G'$ if and only if $x$ is nonadjacent to $y$ in $G$. A \emph{bipartite complement of $G$} is the bipartite complement of $G$ with respect to some bipartition of $G$.

\subsection*{Algorithms}

If $M$ is a matrix, we denote by $\size(M)$ the sum of the number of rows, the number of columns, and the number of ones of $M$. We say that an algorithm taking a matrix $M$ as input is \emph{linear-time} if it runs in $O(\size(M))$ time. In time and space bounds of algorithms taking a graph as input, we denote by $n$ and $m$ the number of vertices and edges of the input graph. We say that an algorithm taking a graph as input is \emph{linear-time} if it runs in $O(n+m)$ time. We assume that input matrices are represented by lists of rows, where each row is represented by a list of the columns having a $1$ in the row. We assume input graphs are represented by adjacency lists. This way, matrices and graphs are represented in $O(\size(M))$ and $O(n+m)$ space, respectively.

\subsection*{Consecutive-ones property and circular-ones property}

A matrix $M$ has the \emph{consecutive-ones property for rows}~\cite{MR0190028} (resp.~\emph{circular-ones property for rows}~\cite{MR0309810}) if there is a linear order $\leqC$ of the columns of $M$ such that each row of $M$ is a linear interval (resp.\ a circular interval) of $\leqC$. If so, $\leqC$ is called a \emph{consecutive-ones order} (resp.\ a \emph{circular-ones order}) of $M$. The \emph{consecutive-ones property for columns} (resp.~\emph{circular-ones property for columns}) is defined analogously by reversing the roles of rows and columns. If no mention is made to rows or columns, we mean the corresponding property for the rows. If the canonical order of the columns of some matrix $M$ is a circular-ones order of $M$, we say that $M$ is a \emph{circular-ones matrix}.

Booth and Lueker~\cite{MR0433962} gave linear-time recognition algorithms for both the consecutive-ones property and the circular-ones property. (In the theorem below, $M_{\langle 1,2,\ldots,i\rangle}$ denotes the matrix consisting of the first $i$ rows of $M$, and is an instance of the notation $M_\rho$ introduced earlier in this section.)

\begin{thm}[\cite{MR0433962}]\label{thm:booth-and-lueker} There is a linear-time algorithm that, given any matrix $M$, outputs either a consecutive-ones order (resp.\ a circular-ones order) of $M$ or the least positive integer $i$ such that $M_{\langle 1,2,\ldots,i\rangle}$ does not have the consecutive-ones (resp.\ circular-ones) property.\end{thm}

\begin{figure}[t!]
\ffigbox[\textwidth]{%
\begin{subfloatrow}
\ffigbox[\FBwidth]{$\begin{pmatrix}
             1 & 1\\
             & 1 & 1\\
             &   & \ddots & \ddots \\
             &   &        &     1 & 1\\
           1 & 0 & \cdots &     0 & 1
         \end{pmatrix}$
}{\caption{$\MI k$ for each $k\geq 3$, where omitted entries are $0$'s}}\qquad
\ffigbox[\FBwidth]{$\begin{pmatrix}
            1 & 1 & 0 & 0 & 0 & 0\\
            0 & 0 & 1 & 1 & 0 & 0\\
            0 & 0 & 0 & 0 & 1 & 1\\
            0 & 1 & 0 & 1 & 0 & 1
          \end{pmatrix}$
}{\caption{$\MIV$}}\qquad
\ffigbox[\FBwidth]{$\begin{pmatrix}
             1 & 1 & 0 & 0 & 0\\
             1 & 1 & 1 & 1 & 0\\
             0 & 0 & 1 & 1 & 0\\
             1 & 0 & 0 & 1 & 1
          \end{pmatrix}$
}{\caption{$\MV$}}
\end{subfloatrow}

}
{\caption{Some Tucker matrices, where $k$ denotes the number of rows}\label{fig:TuckerMatrices}}
\end{figure}

\begin{figure}[t!]
\ffigbox[\textwidth]{%

\begin{subfloatrow}
\ffigbox[\FBwidth]{
 $\begin{pmatrix}
         1 & 1 & 1\\
         1 & 0 & 0\\
         0 & 1 & 0\\
         0 & 0 & 1
      \end{pmatrix}$}
{\caption{$Z_1$}}

\ffigbox[\FBwidth]{
 $\begin{pmatrix}
         1 & 0 & 0\\
         1 & 1 & 0\\
         1 & 1 & 1\\
         0 & 1 & 0
      \end{pmatrix}$}
{\caption{$Z_2$}}\quad

\ffigbox[\FBwidth]{
  $\begin{pmatrix}
         1 & 0 & 0\\
         1 & 1 & 0\\
         1 & 1 & 1\\
         1 & 0 & 1
      \end{pmatrix}$}
{\caption{$Z_3$}}

\ffigbox[\FBwidth]{
  $\begin{pmatrix}
         1 & 1 & 1 & 0\\
         0 & 1 & 1 & 1\\
         0 & 0 & 1 & 0
      \end{pmatrix}$}
{\caption{$Z_4$}}

\ffigbox[\FBwidth]{
$\begin{pmatrix}
1 & 0 & 0 & 1\\
1 & 1 & 0 & 0\\
1 & 1 & 1 & 0\\
0 & 1 & 0 & 0
\end{pmatrix}$}
{\caption{$\ZetaCinco$}}
\end{subfloatrow}

\bigskip

\begin{subfloatrow}
\ffigbox[\FBwidth]{
$\begin{pmatrix}
1 & 1 & 1 & 0\\
1 & 0 & 0 & 1\\
0 & 1 & 0 & 0\\
0 & 0 & 1 & 0
\end{pmatrix}$}
{\caption{$\ZetaSeis$}}

\ffigbox[\FBwidth]{
$\begin{pmatrix}
1 & 1 & 1 & 0\\
1 & 0 & 0 & 1\\
0 & 1 & 0 & 1\\
0 & 0 & 1 & 0\\
\end{pmatrix}$}
{\caption{$\ZetaSiete$}}\quad

\ffigbox[\FBwidth]{
$\begin{pmatrix}
1 & 1 & 1 & 0 & 1\\
0 & 1 & 1 & 1 & 1\\
0 & 0 & 1 & 0 & 0\\
0 & 0 & 0 & 0 & 1
\end{pmatrix}$}
{\caption{$\ZetaOcho$}}

\end{subfloatrow}
}
{\caption{Some small matrices}\label{fig:smallmatrices}}
\end{figure}

Tucker~\cite{MR0295938} characterized the consecutive-ones property by a minimal set of forbidden submatrices, known as \emph{Tucker matrices}. The matrices $\MI k$ for each $k\geq 3$, $\MIV$, and $\MV$, displayed in Figure~\ref{fig:TuckerMatrices}, are some of the Tucker matrices. In~\cite{1611.02216}, we gave an analogous characterization for the circular-ones property. The corresponding set of forbidden submatrices is
\[ \ForbRow=\{a\miop\MIast k\colon\,k\geq 3\mbox{ and }a\in A_k\}\cup\{\MIV,\overline{\MIV},\MVast,\overline{\MVast}\}, \]
where $\MIast k$ and $\MVast$ denote $(\MI k)^*$ and $(\MV)^*$, $A_3=\{000,111\}$ and, for each $k\geq 4$, $A_k$ is the set of all binary bracelets of length $k$. Notice that $001$ and $011$ are binary bracelets of length $3$ but do not belong to $A_3$. A matrix $M$ is a \emph{minimal forbidden submatrix for the circular-ones property} if $M$ is the only submatrix of $M$ not having the circular-ones property.

\begin{thm}[\cite{1611.02216}]\label{thm:circR} A matrix $M$ has the circular-ones property if and only if $M$ contains no matrix in the set $\ForbRow$ as a configuration. Moreover, there is a linear-time algorithm that, given any matrix $M$ not having the circular-ones property, outputs a matrix in $\ForbRow$ contained in $M$ as a configuration. In addition, every matrix in $\ForbRow$ is a minimal forbidden submatrix for the circular-ones property. Hence, for each $M\in\ForbRow$ and each binary sequence $a$ whose length equals the number of rows of $M$, $a\miop M$ represents the same configuration as some matrix in $\ForbRow$. 
\end{thm}

\section{$D$-circular property}\label{sec:Dcirc}

Figure~\ref{fig:smallmatrices} introduces the matrices $Z_1$, $Z_2$, $Z_3$, $Z_4$, $Z_5$, $Z_6$, $Z_7$, and $Z_8$ needed in what follows. (Notice that $Z_4=Z_2\trans$.) The main aim of this section is to give a characterization by minimal forbidden submatrices and a linear-time recognition algorithm for the $D$-circular property. The corresponding set of minimal forbidden submatrices is
\[ \FDCircR^\infty=\FDCircR\cup\bigcup_{k=3}^\infty\{\MIast k,\overline{\MIast k}\} \]
where
\[ \FDCircR=\{\ZetaUnoAst,\ZetaDosAst,\ZetaTresAst,\ZetaCuatroAst,\ZetaCinco,\ZetaCincoTrans,\ZetaSeis,\ZetaSiete,\ZetaOcho,\CoZetaUnoAst,\CoZetaDosAst,\CoZetaCuatroAst,\CoZetaSeis\}. \]
All the matrices in $\FDCircR^\infty$ are displayed explicitly later in Figure~\ref{fig:forb_DCircRinfty} (see Subsection~\ref{ssec:Dcirc-forb}). Notice that each of $\ZetaTresAst$, $\ZetaCinco$, $\ZetaCincoTrans$, $\ZetaSiete$, and $\ZetaOcho$ represents the same configuration as its complement. Hence, for each matrix $M$ in $\FDCircR$, there is some $M'$ in $\FDCircR$ such that $M'$ represents the same configuration as $\overline M$.

This section is organized as follows. In Subsection~\ref{ssec:Dcirc-circR}, we discuss the connection between the $D$-circular property and the circular-ones property. In Subsections~\ref{ssec:QR}, \ref{ssec:UW}, and~\ref{ssec:XY} some auxiliary matrices are shown to contain as a configuration some matrix in $\FDCircR^\infty$ or in $\ForbRow$; these technical results are crucial for the proof of the minimal forbidden submatrix characterization for the $D$-circular property given in Subsection~\ref{ssec:Dcirc-forb}. In Subsection~\ref{ssec:Dcirc-algo}, we give our linear-time recognition algorithm for the $D$-circular property. Finally, in Subsection~\ref{ssec:Dint}, we study the connection between the results about the $D$-circular property obtained along this section and some known results for the $D$-interval property.

\subsection{Connection with the circular-ones property}\label{ssec:Dcirc-circR}

If $M$ is a matrix, we denote by $D(M)$ a matrix that arises from $M$ by adding rows at the bottom as follows: for each nontrivial rows $r$ and $s$ such that $r$ is properly contained in $s$, add a row equal to the set difference $s-r$. (This operator $D(M)$ for matrices $M$ is intimately related but slightly different from the operator $D(H)$ defined in~\cite{MR0437403} and the operator $H^\Subset$ defined in~\cite{MR3573905} for hypergraphs $H$.) As $D(M)$ arises from $M$ by adding rows, we will usually regard the linear orders of the columns of $M$ as linear orders of the columns of $D(M)$ and vice versa. The following fact is immediate consequence of the definitions.

\begin{lem}\label{lem:D(M)} A matrix $M$ has the $D$-circular property if and only if $D(M)$ has the circular-ones property.\end{lem}
\begin{proof} By definition, each $D$-circular order of $M$ is a circular-ones order of $D(M)$. Hence, if $M$ has the $D$-circular property, then $D(M)$ has the circular-ones property. For the converse, suppose that $D(M)$ has the circular-ones property. Thus, there is some circular-ones order $\leqC$ of $D(M)$. Let $r$ and $s$ be two rows of $M$. As $r$ and $s$ are rows of $D(M)$, then $r$ and $s$ are circular intervals of $\leqC$. Hence, $s-r$ is also a circular interval of $\leqC$ unless, perhaps, when $r$ is properly contained in $s$. If $r$ or $s$ is trivial, then $s-r$ is trivial, $s$, or the complement of $r$, all of which are circular intervals of $\leqC$. Thus, we assume, without loss of generality, that $r$ and $s$ are nontrivial. Therefore, if $r$ is properly contained in $r$, then $s-r$ is a circular interval of $\leqC$ because $s-r$ is a row of $D(M)$. We conclude that in all cases $s-r$ is a circular interval of $\leqC$. Thus, by definition, $\leqC$ is a $D$-circular order of $M$. This proves that, whenever $D(M)$ has the circular-ones property, $M$ has the $D$-circular property. The proof of the lemma is complete.\end{proof}

By virtue of the above lemma and the fact that complementing some rows of a matrix preserves the circular-ones property, the following result implies that $M$ has the $D$-circular property if and only if $\overline M$ has the $D$-circular property.

\begin{lem}\label{lem:D(co-M)} If $M$ is a $k\times\ell$ matrix, then $D(\overline M)$ arises from $D(M)$ by complementing its first $k$ rows and permuting the remaining rows.\end{lem}
\begin{proof} Let $M'$ be the matrix that arises from $D(M)$ by complementing its first $k$ rows. By definition of $D(M)$, the first $k$ rows of $M'$ coincide with the first $k$ rows of $D(\overline M)$. Let $i,j\in[k]$. Notice that $\overline M_{\langle i\rangle}$ is properly contained in $\overline M_{\langle j\rangle}$ if and only if $M_{\langle j\rangle}$ is properly contained in $M_{\langle i\rangle}$. Moreover, $\overline M_{\langle j\rangle}-\overline M_{\langle i\rangle}=M_{\langle i\rangle}-M_{\langle j\rangle}$. We conclude that the rows of $D(\overline M)$ and $M'$ are the same up to permutation. This completes the proof of the lemma.\end{proof}

As none of the matrices in $\ForbRow$ has the circular-ones property (see Theorem~\ref{thm:circR}), Lemma~\ref{lem:D(M)} together with the result below shows that none of the matrices in $\FDCircR^\infty$ has the $D$-circular property.

\begin{lem}\label{lem:forb_D} If $F\in\FDCircR^\infty$, then $D(F)$ contains some matrix in $\ForbRow$ as a configuration.\end{lem}
\begin{proof} If $F$ is $\MIast k$ or $\overline{\MIast k}$ for some $k\geq 3$, then the lemma holds immediately because $D(F)$ contains $F$ as a configuration and $F\in\ForbRow$. Moreover, the following assertions can be verified by inspection: (i) if $F\in\{\ZetaUnoAst,\ZetaDosAst,\ZetaTresAst,\ZetaSiete,\CoZetaUnoAst,\CoZetaSeis\}$, then $D(F)$ contains $\MIast 3$ as a configuration; (ii) if $F\in\{\ZetaCinco,\ZetaCincoTrans,\ZetaSeis,\CoZetaDosAst\}$, then $D(F)$ contains $\overline{\MIast 3}$ as a configuration; (iii) if $F=\ZetaCuatroAst$, then $D(F)$ contains $0011\miop\MIast 4$ as a configuration; (iv) if $F=\CoZetaCuatroAst$, then $D(F)$ contains $\MIast 4$ as a configuration; (v) if $F=\ZetaOcho$, then $D(F)$ contains $\overline{\MIast 4}$ as a configuration. Since $\MIast 3$, $\overline{\MIast 3}$, $0011\miop\MIast 4$, $\MIast 4$, and $\overline{\MIast 4}$ belong to $\ForbRow$, the proof of the lemma is complete.\end{proof}

Our next result shows that each matrix in $\ForbRow$ contains some matrix in $\FDCircR^\infty$ as a configuration.

\begin{lem}\label{lem:forb_Circ1R} If $F\in\ForbRow$, then $F$ contains some matrix in $\FDCircR^\infty$ as a configuration.\end{lem}
\begin{proof} If $F$ is $\MIV$ or $\overline{\MIV}$, then $F_{\langle 4,1,2,3\rangle,\langle 2,4,6,1\rangle}$ is $\ZetaSeis$ or $\overline{\ZetaSeis}$, respectively. If $F$ is $\MVast$ or $\overline{\MVast}$, then $F_{\langle 1,2,4,3\rangle,\langle 1,4,5,6\rangle}$ is $\ZetaDosAst$ or $\overline{\ZetaDosAst}$, respectively. Thus, it only remains to consider the case where $F=a\miop\MIast k$ for some $k\geq 3$ and some binary sequence $a=a_1a_2\ldots a_k$ such that $a\in A_k$. If $a$ consists entirely of $0$'s or entirely of $1$'s, then $F$ coincides with $\MIast k$ or $\overline{\MIast k}$, both of which belong to $\FDCircR^\infty$. Hence, we assume, without loss of generality, that $a$ is nonconstant (i.e, $a$ contains at least a $0$ and at least a $1$) and, necessarily, $k\geq 4$.

Suppose first that $a_{i+3}\neq a_i$ for some $i\in[k]$ (where subindices are modulo $k$). Thus, some sequence $b=b_1b_2b_3b_4$ of length $4$ such that $b_1\neq b_4$ occurs circularly in $a$ at position $i$. If we let $\rho=\langle i,i+1,i+2,i+3\rangle$ and $\sigma=\langle i+1,i+2,i+3,k+1\rangle$ (where sums involving $i$ are modulo $k$), then $a_\rho=b$ and
\[ F_{\rho,\sigma}=(a\miop\MIast k)_{\rho,\sigma}
                  =a_\rho\miop\MIast k_{\rho,\sigma}
                  =b\miop\begin{pmatrix}
                     1 & 0 & 0 & 0\\
                     1 & 1 & 0 & 0\\
                     0 & 1 & 1 & 0\\
                     0 & 0 & 1 & 0
                    \end{pmatrix}. \]
On the one hand, if $b_1=0$, then $b$ is $0001$, $0011$, $0101$, or $0111$ and, consequently, $(F_{\rho,\sigma})_{\langle 3,2,4,1\rangle,\langle 2,1,4,3\rangle}=\ZetaCinco$, $(F_{\rho,\sigma})_{\langle 1,3,4,2\rangle,\langle 1,4,2,3\rangle}=\ZetaTresAst$, $(F_{\rho,\sigma})_{\langle 4,3,2,1\rangle,\langle 2,4,1,3\rangle}=\ZetaSiete$, or $(F_{\rho,\sigma})_{\langle 2,3,4,1\rangle,\langle 4,1,2,3\rangle}=\ZetaCinco$, respectively. On the other hand, if $b_1=1$, then, since $b\miop\MIast k_{\rho,\sigma}=\overline{\overline b\miop\MIast k_{\rho,\sigma}}$ and each of $\ZetaTresAst$, $\ZetaCinco$, and $\ZetaSiete$ represents the same configuration as its complement, the matrix $F_{\rho,\sigma}$ represents the same configuration as $\ZetaTresAst$, $\ZetaCinco$, or $\ZetaSiete$.

Suppose now that $a_{i+3}=a_i$ for each $i\in[k]$ (where subindices are modulo $k$). Since $a$ is a nonconstant bracelet, its prefix of length $3$ must be $001$, $010$, or $011$. Thus, $k$ is a multiple of $3$ and $a$ is the concatenation of $k/3$ copies of that prefix. As a consequence, some sequence $b\in\{01001,10110\}$ occurs circularly in $a$ at position $i$ for some $i\in[k]$. If we let $\rho=\langle i,i+1,i+2,i+4\rangle$ and $\sigma=\langle i+1,i+2,i+4,k+1\rangle$ (where the sums involving $i$ are modulo $k$), then $a_\rho=b_{\langle 1,2,3,5\rangle}$,
\[ F_{\rho,\sigma}=(a\miop\MIast k)_{\rho,\sigma}
                  =a_\rho\miop\MIast k_{\rho,\sigma}
                  =b_{\langle 1,2,3,5\rangle}\miop\begin{pmatrix}
                                                               1 & 0 & 0 & 0\\
                                                               1 & 1 & 0 & 0\\
                                                               0 & 1 & 0 & 0\\
                                                               0 & 0 & 1 & 0
                                                   \end{pmatrix}, \]
and, consequently, $(F_{\rho,\sigma})_{\langle 4,2,1,3\rangle,\langle 4,1,2,3\rangle}$ equals $\ZetaSeis$ or $\CoZetaSeis$ depending whether $b$ is $01001$ or $10110$, respectively. This completes the proof of the lemma.\end{proof}

\subsection{Matrices $Q$ and $R$}\label{ssec:QR}

We will associate with each senary sequence $\lambda$ of length at least $3$ a matrix denoted $R(\lambda)$. As a preliminary result for the proof of Theorem~\ref{thm:main}, we need to prove Lemma~\ref{lem:R} below, which asserts that, for almost all senary sequences $\lambda$ of length at least $3$, $R(\lambda)$ contains some matrix in $\FDCircR^\infty$ as a configuration. We now introduce the necessary definitions.

For each $k\geq 3$ and each $i\in[k]$, we define the following matrices, where in all the cases $i+1$ should be understood modulo $k$:
\begin{itemize}
 \item $Q_0(i,k)$ is the $1\times(k+1)$ matrix whose only row has $1$'s at columns $i$ and $i+1$ and $0$'s at the remaining ones;

 \item $Q_1(i,k)$ is the complement of $Q_0(i,k)$;
 
 \item $Q_2(i,k)$ is the $2\times(k+1)$ matrix whose first row has a $0$ at column $k+1$ and $1$'s at the remaining columns and whose second row has $0$'s at columns $i$, $i+1$, and $k+1$ and $1$'s at the remaining columns;
 
 \item $Q_3(i,k)$ is the complement of $Q_2(i,k)$;
 
 \item $Q_4(i,k)$ is the $2\times(k+1)$ matrix whose first row has a $0$ at column $i+1$ and $1$'s at the remaining columns and whose second row has a $1$ at column $i$ and $0$'s at the remaining columns;

 \item $Q_5(i,k)$ is the complement of $Q_4(i,k)$.
\end{itemize}

Given a senary sequence $\lambda=\lambda_1\lambda_2\ldots\lambda_k$ of length $k$ for some $k\geq 3$, we denote by $R(\lambda)$ the matrix having $k+1$ columns and whose rows are those of $Q_{\lambda_1}(1,k)$, followed by those of $Q_{\lambda_2}(2,k)$, followed by those of $Q_{\lambda_3}(3,k)$, \ldots, followed by those of $Q_{\lambda_k}(k,k)$. For instance,
\[ R(012345)=
\begin{pmatrix}
1 & 1 & 0 & 0 & 0 & 0 & 0\\
1 & 0 & 0 & 1 & 1 & 1 & 1\\
1 & 1 & 1 & 1 & 1 & 1 & 0\\
1 & 1 & 0 & 0 & 1 & 1 & 0\\
0 & 0 & 0 & 0 & 0 & 0 & 1\\
0 & 0 & 0 & 1 & 1 & 0 & 1\\
1 & 1 & 1 & 1 & 1 & 0 & 1\\
0 & 0 & 0 & 0 & 1 & 0 & 0\\
1 & 0 & 0 & 0 & 0 & 0 & 0\\
1 & 1 & 1 & 1 & 1 & 0 & 1
\end{pmatrix}. \]

\newcommand\LemR{Let $\lambda=\lambda_1\lambda_2\ldots\lambda_k$ be a senary sequence of length $k$ such that $k\geq 3$. If $k=3$, suppose additionally that neither $4$ nor $5$ occurs in $\lambda$. Then, $R(\lambda)$ contains some matrix in $\FDCircR^\infty$ as a configuration.}
\begin{lem}\label{lem:R}\LemR\end{lem}

The proof of Lemma~\ref{lem:R} is given in Section~\ref{app:lem:R} of the appendix.

\subsection{Matrices $U$ and $W$}\label{ssec:UW}

We will now associate with some senary sequences $\lambda$ of length $4$, a corresponding matrix $W(\lambda)$ and we will show that, for certain such sequences $\lambda$, $W(\lambda)$ contains some matrix in $\FDCircR$ as a configuration (see Lemma~\ref{lem:W}).

We first define, for each $i\in[4]$, the following matrices:
\begin{itemize}
 \item $U_0(i)$ whose only rows coincides with row $i$ of $\MIV$;
 \item $U_1(i)$ is the complement of $U_0(i)$.
\end{itemize}
For each $i\in[3]$, we define the following matrices, where sums involving $i$ are modulo $3$:
\begin{itemize}
 \item $U_2(i)$ is the $2\times 6$ matrix whose first row coincides with the complement of row $i+1$ of $\MIV$ and the second row coincides with row $i+2$ of $\MIV$;
 \item $U_3(i)$ is the complement of $U_2(i)$.
\end{itemize}
We need two sporadic matrices:
\begin{itemize}
 \item $U_4(3)=\begin{pmatrix}
                  1 & 1 & 1 & 1 & 1 & 0\\
                  0 & 0 & 0 & 0 & 1 & 0
               \end{pmatrix}$;
 \item $U_5(3)$ is the complement of $U_4(3)$.
\end{itemize}

For each senary sequence $\lambda=\lambda_1\lambda_2\lambda_3\lambda_4$ of length $4$ such that $\lambda_1,\lambda_2\in\{0,1,2,3\}$ and $\lambda_4\in\{0,1\}$, we define $W(\lambda)$ as the matrix having six columns and whose rows are those of $U_{\lambda_1}(1)$, followed those of $U_{\lambda_2}(2)$, followed by those of $U_{\lambda_3}(3)$, followed by those of $U_{\lambda_4}(4)$. For instance,
\[ W(0341) =
\begin{pmatrix}
1 & 1 & 0 & 0 & 0 & 0\\
0 & 0 & 0 & 0 & 1 & 1\\
0 & 0 & 1 & 1 & 1 & 1\\
1 & 1 & 1 & 1 & 1 & 0\\
0 & 0 & 0 & 0 & 1 & 0\\
1 & 0 & 1 & 0 & 1 & 0
\end{pmatrix}. \]

\newcommand\LemW{If $\lambda=\lambda_1\lambda_2\lambda_3\lambda_4$ is a senary sequence such that $\lambda_1,\lambda_2\in\{0,1,2,3\}$ and $\lambda_4\in\{0,1\}$, then $W(\lambda)$ contains some matrix in $\FDCircR$ as a configuration.}
\begin{lem}\label{lem:W}\LemW\end{lem}

The proof of Lemma~\ref{lem:W} is given in Section~\ref{app:lem:W} of the appendix.

\subsection{Matrices $X$ and $Y$}\label{ssec:XY}

For each binary sequence $\alpha=\alpha_1\alpha_2\alpha_3\alpha_4$ of length $4$ and each $i\in[3]$, we define $X_i(\alpha)$ as the $6\times 6$ matrix that arises from $\MIV$ by adding a fifth row having $1$'s in columns $2i-1$ and $2i$, and a sixth row having $0$'s in columns $2i-1$ and $2i$ and such that the entries at each of the columns $2i+1$, $2i+2$, $2i+3$, and $2i+4$ (where additions are modulo $6$) coincide in the fifth and sixth rows and are equal to $\alpha_1$, $\alpha_2$, $\alpha_3$, and $\alpha_4$, respectively. For instance,
\[ X_1(\alpha_1\alpha_2\alpha_3\alpha_4)=
     \begin{pmatrix}
       1 & 1 & 0 & 0 & 0 & 0\\
       0 & 0 & 1 & 1 & 0 & 0\\
       0 & 0 & 0 & 0 & 1 & 1\\
       0 & 1 & 0 & 1 & 0 & 1\\
       1 & 1 & \alpha_1 & \alpha_3 & \alpha_3 & \alpha_4 \\
       0 & 0 & \alpha_1 & \alpha_2 & \alpha_3 & \alpha_4
     \end{pmatrix}. \]

For each binary sequence $\gamma=\gamma_1\gamma_2\gamma_3$ of length $3$, we define
\[ Y(\gamma)=
    \begin{pmatrix}
      1 & 1 & 0 & 0 & 0 & 0\\
      0 & 0 & 1 & 1 & 0 & 0\\
      0 & 0 & 0 & 0 & 1 & 1\\
      0 & 1 & 0 & 1 & 0 & 1\\
      \gamma_1 & 1 & \gamma_2 & 1 & \gamma_3 & 1\\
      \gamma_1 & 0 & \gamma_2 & 0 & \gamma_3 & 0
    \end{pmatrix}. \]

\newcommand\LemX{Let $\alpha$ be a binary sequence of length $4$ and let $i\in[3]$. If $\alpha\notin\{0000,0011,1100,1111\}$, then $X_i(\alpha)$ contains as a configuration some matrix in $\ForbRow$ having fewer than $6$ columns.}
\begin{lem}\label{lem:X}\LemX\end{lem}

\newcommand\LemY{Let $\gamma$ be a binary sequence of length $3$. If $\gamma$ is nonconstant, then $Y(\gamma)$ contains as a configuration some matrix in $\ForbRow$ having fewer than $6$ columns.}
\begin{lem}\label{lem:Y}\LemY\end{lem}

The proofs of Lemmas~\ref{lem:X} and~\ref{lem:Y} are given in Sections~\ref{app:lem:X} and~\ref{app:lem:Y}, respectively, of the appendix.

\subsection{Forbidden submatrix characterization of the $D$-circular property}\label{ssec:Dcirc-forb}

The following is the main structural result of this work and gives a minimal forbidden submatrix characterization of the circularly compatible ones property. Recall the definitions of $\FDCircR$ and $\FDCircR^\infty$ given at the beginning of this section. For the readers' sake, all the matrices in $\FDCircR^\infty$ are displayed explicitly in Figure~\ref{fig:forb_DCircRinfty}.

\begin{figure}[t!]
\ffigbox[\textwidth]{%

\begin{subfloatrow}
\ffigbox[\FBwidth]{
 $\begin{pmatrix}
         1 & 1 & 1 & 0\\
         1 & 0 & 0 & 0\\
         0 & 1 & 0 & 0\\
         0 & 0 & 1 & 0
      \end{pmatrix}$}
{\caption{$\ZetaUnoAst$}}

\ffigbox[\FBwidth]{
 $\begin{pmatrix}
         1 & 0 & 0 & 0\\
         1 & 1 & 0 & 0\\
         1 & 1 & 1 & 0\\
         0 & 1 & 0 & 0
      \end{pmatrix}$}
{\caption{$\ZetaDosAst$}}\quad

\ffigbox[\FBwidth]{
  $\begin{pmatrix}
         1 & 0 & 0 & 0\\
         1 & 1 & 0 & 0\\
         1 & 1 & 1 & 0\\
         1 & 0 & 1 & 0
      \end{pmatrix}$}
{\caption{$\ZetaTresAst$}}

\ffigbox[\FBwidth]{
  $\begin{pmatrix}
         1 & 1 & 1 & 0 & 0\\
         0 & 1 & 1 & 1 & 0\\
         0 & 0 & 1 & 0 & 0
      \end{pmatrix}$}
{\caption{$\ZetaCuatroAst$}}

\ffigbox[\FBwidth]{
$\begin{pmatrix}
1 & 0 & 0 & 1\\
1 & 1 & 0 & 0\\
1 & 1 & 1 & 0\\
0 & 1 & 0 & 0
\end{pmatrix}$}
{\caption{$\ZetaCinco$}}
\end{subfloatrow}

\bigskip

\begin{subfloatrow}
\ffigbox[\FBwidth]{
$\begin{pmatrix}
1 & 1 & 1 & 0\\
0 & 1 & 1 & 1\\
0 & 0 & 1 & 0\\
1 & 0 & 0 & 0
\end{pmatrix}$}
{\caption{$\ZetaCincoTrans$}}

\ffigbox[\FBwidth]{
$\begin{pmatrix}
1 & 1 & 1 & 0\\
1 & 0 & 0 & 1\\
0 & 1 & 0 & 0\\
0 & 0 & 1 & 0
\end{pmatrix}$}
{\caption{$\ZetaSeis$}}\quad

\ffigbox[\FBwidth]{
$\begin{pmatrix}
1 & 1 & 1 & 0\\
1 & 0 & 0 & 1\\
0 & 1 & 0 & 1\\
0 & 0 & 1 & 0\\
\end{pmatrix}$}
{\caption{$\ZetaSiete$}}

\ffigbox[\FBwidth]{
$\begin{pmatrix}
1 & 1 & 1 & 0 & 1\\
0 & 1 & 1 & 1 & 1\\
0 & 0 & 1 & 0 & 0\\
0 & 0 & 0 & 0 & 1
\end{pmatrix}$}
{\caption{$\ZetaOcho$}}

\ffigbox[\FBwidth]{
$\begin{pmatrix}
0 & 0 & 0 & 1\\
0 & 1 & 1 & 1\\
1 & 0 & 1 & 1\\
1 & 1 & 0 & 1
\end{pmatrix}$}
{\caption{$\CoZetaUnoAst$}}

\end{subfloatrow}

\bigskip

\begin{subfloatrow}
\ffigbox[\FBwidth]{
$\begin{pmatrix}
0 & 1 & 1 & 1\\
0 & 0 & 1 & 1\\
0 & 0 & 0 & 1\\
1 & 0 & 1 & 1
\end{pmatrix}$}
{\caption{$\CoZetaDosAst$}}

\ffigbox[\FBwidth]{
$\begin{pmatrix}
0 & 0 & 0 & 1 & 1\\
1 & 0 & 0 & 0 & 1\\
1 & 1 & 0 & 1 & 1
\end{pmatrix}$}
{\caption{$\CoZetaCuatroAst$}}\quad

\ffigbox[\FBwidth]{
$\begin{pmatrix}
0 & 0 & 0 & 1\\
0 & 1 & 1 & 0\\
1 & 0 & 1 & 1\\
1 & 1 & 0 & 1
\end{pmatrix}$}
{\caption{$\CoZetaSeis$}}

\ffigbox[\FBwidth]{
$\begin{pmatrix}
           1 & 1 &        &       &   & 0\\
             & 1 & 1      &       &   & 0\\
             &   & \ddots & \ddots &  & \vdots \\
             &   &        &     1 & 1 & 0\\
           1 & 0 & \cdots &     0 & 1 & 0
\end{pmatrix}$
}{\caption{$\MIast k$ for each $k\geq 3$, where omitted entries are $0$'s}}

\ffigbox[\FBwidth]{
$\begin{pmatrix}
           0 & 0 &        &       &   & 1\\
             & 0 & 0      &       &   & 1\\
             &   & \ddots & \ddots &  & \vdots \\
             &   &        &     0 & 0 & 1\\
           0 & 1 & \cdots &     1 & 0 & 1
\end{pmatrix}$
}{\caption{$\overline{\MIast k}$ for each $k\geq 3$, where omitted entries are $1$'s}}

\end{subfloatrow}
}
{\caption{Full list of matrices in $\FDCircR^\infty$ where $k$ denotes the number of rows}\label{fig:forb_DCircRinfty}}
\end{figure}

\begin{thm}\label{thm:main} For each matrix $M$, all the following assertions are equivalent:
\begin{enumerate}[(i)]
 \item\label{it:main1} $M$ has the $D$-circular property;
 \item\label{it:main2} $M$ has the circular-ones property and $M$ contains no matrix in $\FDCircR$ as a configuration;
 \item\label{it:main3} $M$ contains no matrix in $\FDCircR^\infty$ as a configuration.
\end{enumerate}\end{thm}
\begin{proof} That assertion~\eqref{it:main1} implies assertion~\eqref{it:main2} follows from Lemmas~\ref{lem:D(M)} and~\ref{lem:forb_D} because the circular-ones property of a matrix $M$ is inherited by any matrix that is contained in $M$ as a configuration and the same holds for the $D$-circular property. That assertion~\eqref{it:main2} implies assertion~\eqref{it:main3} follows from the fact that none of $\MIast k$ and $\overline{\MIast k}$ has the circular-ones property for any $k\geq 3$.

It only remains to prove that assertion~\eqref{it:main3} implies assertion~\eqref{it:main1}. Suppose, for a contradiction, that assertion~\eqref{it:main3} does not imply assertion~\eqref{it:main1} and let $M$ be a matrix having the minimum possible number of columns such that $M$ contains no matrix in $\FDCircR^\infty$ as a configuration and $M$ does not have the $D$-circular property. Let $k_M$ and $\ell_M$ be the number of rows and columns of $M$, respectively, and let $k_D$ be the number of rows of $D(M)$. By Lemma~\ref{lem:D(M)}, $D(M)$ does not have the circular-ones property. By Theorem~\ref{thm:circR}, $D(M)$ contains some matrix in $\ForbRow$ as a configuration. Let $F$ be a matrix in $\ForbRow$ having the minimum possible number of columns such that $F$ is contained in $D(M)$ as a configuration. Let $\rho$ and $\sigma$ be a row map and a column map such that $D(M)_{\rho,\sigma}=F$. Thus, $(D(M)_{\id_{k_D},\sigma})_\rho=F$. Since $D(M)_{\id_{k_D},\sigma}$ represents the same configuration as some matrix that arises from $D(M_{\id_{k_M},\sigma})$ by the addition of some  (eventually zero) empty rows and $F$ has no empty rows, there is some row map $
\rho'$ such that $(D(M_{\id_{k_M},\sigma}))_{\rho'}=F$. By Lemma~\ref{lem:D(M)}, $M_{\id_{k_M},\sigma}$ does not have the $D$-circular property. We replace $M$ by $M_{\id_{k_M},\sigma}$ and $\rho$ by $\rho'$. Thus, $D(M)_\rho=F$. Hence, it still holds that $M$ is a matrix having the minimum possible number of columns such that $M$ contains no matrix in $\FDCircR^\infty$ as a configuration and $M$ does not have the $D$-circular property. Let $\tilde F$ be the matrix in $\ForbRow$ that represents the same configuration as $\overline F$ (see Theorem~\ref{thm:circR}). Notice that, if necessary, we may further replace the matrices $M$ and $F$ by the matrices $\overline M$ and $\tilde F$, respectively. In fact, $\overline M$ is also a matrix having the minimum possible number of columns such that $\overline M$ contains no matrix in $\FDCircR^\infty$ as a configuration and $\overline M$ does not have the $D$-circular property (by virtue of Lemmas~\ref{lem:D(M)} and~\ref{lem:D(co-M)}), and $\overline M$ contains $\tilde F$ as a configuration, where $\tilde F\in\ForbRow$ and has the minimum possible number of columns among the matrices in $\ForbRow$ contained in $D(\overline M)$ as a configuration (because any matrix that arises from a matrix in $\ForbRow$ by complementing some rows represents the same configuration as some matrix in $\ForbRow$; see Theorem~\ref{thm:circR}).

Suppose first that $F=a\miop\MIast k$ for some $k\geq 3$ and some $a\in A_k$. (Notice that since $M$ and $F$ have the same number of columns, necessarily $k+1=\ell_M$.) By replacing $M$ and $F$ by $\overline M$ and $\tilde F$ (if necessary), we assume, without loss of generality, that if $k=3$ then $a=000$. (Recall that $A_3=\{000,111\}$.) For each $i\in[k]$ and each $j\in\{0,1,2,3,4,5\}$, we define the statement $S_{i,j}$ as follows:
\begin{enumerate}
 \item[$(S_{i,j})$] Each row of $Q_j(i,k)$ coincides with some row of $M$.
\end{enumerate}
We will prove that, for each $i\in[k]$, $S_{i,j}$ holds for some $j\in\{0,1,2,3,4,5\}$. More precisely, for each $i\in[k]$, the following claims hold (where $i+1$ should be understood modulo $k$):
\begin{enumerate}[{Claim }1:]
\item\label{claim:T1} \emph{If $a_i=0$, then $S_{i,0}$, $S_{i,1}$, $S_{i,2}$, or $S_{i,3}$ holds.} Suppose that $a_i=0$ and let $u=\rho(i)$. If $u\in[k_M]$, then $S_{i,0}$ holds because $M_{\langle u\rangle}=D(M)_{\langle u\rangle}=F_{\langle i\rangle}=Q_0(i,k)$. Suppose otherwise and let $r,s\in[k_M]$ such that $M_{\langle r\rangle}$ is properly contained in $M_{\langle s\rangle}$ and $D(M)_{\langle u\rangle}=M_{\langle s\rangle}-M_{\langle r\rangle}$. Since $D(M)_{\langle u\rangle}$ has $1$'s in columns $i$ and $i+1$ and $0$'s in the remaining columns, necessarily $M_{\langle r\rangle}$ and $M_{\langle s\rangle}$ coincide in all the columns except for the columns $i$ and $i+1$, in each of which $M_{\langle r\rangle}$ has a $0$ and $M_{\langle s\rangle}$ has a $1$. We consider first the case where the entry at column $k+1$ of $M_{\langle r\rangle}$ (and thus also of $M_{\langle s\rangle}$) is $0$. Suppose, for a contradiction, that among the entries of $M_{\langle s\rangle}$ in columns different from $i$, $i+1$, and $k+1$, there is some $0$ and there is some $1$. Thus, there are $j,h\in[k]$ such that $M_{\langle s\rangle}$ has $0$'s in all the columns in $[j,h]_{[k]}$ and has $1$'s in columns $j-1$ and $h+1$ (where the sum and the subtraction are modulo $k$) where $(j-1,h+1)\neq(i+1,i)$. Hence, if $a'$ is the sequence that arises from $a_{\rho\circ\langle j-1,j,j+1\ldots,h\rangle}$ by appending a $0$, $\rho'=\langle\rho(j-1),\rho(j),\rho(j+1),\ldots,\rho(h),s\rangle$, and $\sigma'=\langle j-1,j,j+1,\ldots,h+1,k+1\rangle$ 
(where sums and subtractions are modulo $k$), then $D(M)_{\rho',\sigma'}=a'\miop\MIast{\vert a'\vert}$ is a matrix representing the same configuration as some matrix in $\ForbRow$ (see Theorem~\ref{thm:circR}) that is contained in $D(M)$ as a configuration and has $\vert a'\vert+1<k+1=\ell_M$ columns, contradicting the choice of $F$. This contradiction shows that $M_{\langle s\rangle}$ has all $0$'s or all $1$'s in the columns different from $i$, $i+1$, and $k+1$. Hence, either $S_{i,0}$ holds because $M_{\langle s\rangle}$ coincides with $Q_0(i,k)$ or $S_{i,2}$ holds because $M_{\langle s\rangle}$ and $M_{\langle r\rangle}$ coincide with the rows of $Q_2(i,k)$. The case where the entry at column $k+1$ of $M_{\langle r\rangle}$ (and thus also of $M_{\langle s\rangle}$) is $1$ can be handled symmetrically (by reversing the roles $M_{\langle r\rangle}$ and $M_{\langle s\rangle}$ and of the $0$'s and $1$'s in the analysis) in order to prove that either $S_{i,1}$ holds (because $M_{\langle r\rangle}$ coincides with $Q_1(i,k)$) or $S_{i,3}$ holds (because $M_{\langle r\rangle}$ and $M_{\langle s\rangle}$ coincide with the rows of $Q_3(i,k)$). This completes the proof of the claim.

\item\label{claim:T2} \emph{If $a_i=1$, then $S_{i,0}$, $S_{i,1}$, $S_{i,4}$, or $S_{i,5}$ holds.} Suppose $a_i=1$ and let $u=\rho(i)$. If $u\in[k_M]$, then $S_{i,1}$ holds because $M_{\langle u\rangle}=D(M)_{\langle u\rangle}=\overline{F_{\langle i\rangle}}=Q_1(i,k)$. Suppose otherwise and let $r,s\in[k_M]$ such that $M_{\langle r\rangle}$ is properly contained in $M_{\langle s\rangle}$ and $D(M)_{\langle u\rangle}=M_{\langle s\rangle}-M_{\langle r\rangle}$. Since $D(M)_{\langle u\rangle}$ has $0$'s in columns $i$ and $i+1$ and $1$'s in the remaining columns, necessarily $M_{\langle s\rangle}$ has $1$'s in all the columns different from $i$ and $i+1$, $M_{\langle r\rangle}$ has $0$'s in all the columns different from $i$ and $i+1$, and $M_{\langle r\rangle}$ and $M_{\langle s\rangle}$ coincide in columns $i$ and $i+1$. Let $x$ and $y$ be the common entries of $M_{\langle r\rangle}$ and $M_{\langle s\rangle}$ at columns $i$ and $i+1$, respectively. If $(x,y)=(0,0)$, then $S_{i,1}$ holds because $M_{\langle s\rangle}$ coincides with $Q_1(i,k)$. If $(x,y)=(1,1)$, then $S_{i,0}$ holds because $M_{\langle r\rangle}$ coincides with $Q_0(i,k)$. Otherwise, $S_{i,4}$ or $S_{i,5}$ holds because $M_{\langle s\rangle}$ and $M_{\langle r\rangle}$ coincide with the rows of $Q_4(i,k)$ or the rows of $Q_5(i,k)$, depending on whether $(x,y)$ is $(1,0)$ or $(0,1)$, respectively. This completes the proof of the claim.
\end{enumerate}

Because of the above claims, there is some senary sequence $\lambda=\lambda_1\lambda_2\ldots\lambda_k$ such that $S_{i,\lambda_i}$ holds for each $i\in[k]$. If $k=3$, we choose $\lambda$ in such a way that neither $4$ nor $5$ occurs in $\lambda$ (which is possible by virtue of Claim~\ref{claim:T1} because we are assuming that if $k=3$ then $a=000$). Hence, by Lemma~\ref{lem:R}, $R(\lambda)$ contains some matrix $F'$ in $\FDCircR^\infty$ as a configuration; in particular $F'$ has pairwise different rows. Since, by construction, each row of $R(\lambda)$ coincides with some row $M$, $M$ contains $F'$ as a configuration. This contradicts the assumption that $M$ contains no matrix in $\FDCircR^\infty$ as a configuration. This contradiction proves that $F\neq a\miop\MIast k$ for every $k\geq 3$ and every $a\in A_k$. As $F\in\ForbRow$, necessarily $F\in\{\MIV,\overline{\MIV},\MVast,\overline{\MVast}\}$. Because of the choice of $F$, $D(M)$ contains as a configuration no matrix in $\ForbRow$ having fewer than $6$ columns.

By replacing the matrices $M$ and $F$ by the matrices $\overline M$ and $\tilde F$ (if necessary), we assume, without loss of generality, that $F\in\{\MIV,\MVast\}$. Hence, there is some $a\in\{0000,0010\}$ such that $D(M)$ contains $a\miop\MIV$ as a configuration (because $\MVast$ represents the same configuration as $0010\miop\MIV$). Let $\rho'$ and $\sigma'$ be a row and a column map such that $D(M)_{\rho',\sigma'}=a\miop\MIV$. Reasoning as we did earlier, there is some row map $\rho''$ such that $D(M_{\id_{k_M},\sigma'})_{\rho''}=a\miop\MIV$. We replace $M$ by $M_{\id_{k_M},\sigma'}$ and $\rho$ by $\rho''$. Thus, $D(M)_\rho=a\miop\MIV$ and it still holds that $M$ contains no matrix in $\FDCircR^\infty$ as a configuration and $D(M)$ contains as a configuration no matrix in $\ForbRow$ having fewer than $6$ columns.

For each $i\in\{1,2,3,4\}$ and each $j\in\{0,1,2,3,4,5\}$ satisfying the following two conditions:
\begin{itemize}
 \item if $i\in\{1,2\}$, then $j\in\{0,1,2,3\}$, and 
 \item if $i=4$, then $j\in\{0,1\}$, 
\end{itemize}
we define the following statement:
\begin{enumerate}
 \item[$(T_{i,j})$] Each row of $U_j(i)$ coincides with some row of $M$.
\end{enumerate}

We will prove that for each $i\in[4]$, there is some $j\in\{0,1,2,3,4,5\}$ such that $T_{i,j}$ holds. More precisely, the following claims hold:
\begin{enumerate}[{Claim }1:]
 \item \emph{If $a_i=0$ for some $i\in[3]$, then $T_{i,0}$, $T_{i,1}$, $T_{i,2}$, or $T_{i,3}$ holds.} Suppose $a_i=0$ and let $u=\rho(i)$. If $u\in[k_M]$, then $T_{i,0}$ holds because $M_{\langle u\rangle}=D(M)_{\langle u\rangle}=U_0(i)$. Suppose otherwise and let $r,s\in[k_M]$ such that $M_{\langle r\rangle}$ is a properly contained in $M_{\langle s\rangle}$ and $D(M)_{\langle u\rangle}=M_{\langle s\rangle}-M_{\langle r\rangle}$. Since $D(M)_{\langle u\rangle}$ has $1$'s in columns $2i-1$ and $2i$ and $0$'s in the remaining columns, necessarily $M_{\langle s\rangle}$ has $1$'s in columns $2i-1$ and $2i$, $M_{\langle r\rangle}$ has $0$'s in columns $2i-1$ and $2i$, and $M_{\langle r\rangle}$ and $M_{\langle s\rangle}$ coincide in all the remaining columns. Let $\alpha_1,\alpha_2,\alpha_3,\alpha_4$ be the entries of $M_{\langle r\rangle}$ (and thus also of $M_{\langle s\rangle}$) at columns $2i+1$, $2i+2$, $2i+3$, and $2i+4$ (where sums are modulo $6$). Let $\alpha=\alpha_1\alpha_2\alpha_3\alpha_4$. Notice that each row of $X_i(\alpha)$ coincides with some row of $D(M)$ or the complement of some row of $D(M)$. Suppose, for a contradiction, that $\alpha\notin\{0000,1111,0011,1100\}$. Lemma~\ref{lem:X} ensures that $X_i(\alpha)$ contains as a configuration some matrix $F'$ in $\ForbRow$ having fewer than $6$ columns; in particular, $F'$ has pairwise different rows. Hence, $D(M)$ contains as a configuration some matrix $F''$ that arises from $F'$ by complementing some (eventually zero) rows. By Theorem~\ref{thm:circR}, $F''$ represents the same configuration as some matrix in $\ForbRow$. As $F''$ has fewer than $6$ columns, we reach a contradiction with the choice of $F$. This contradiction proves that $\alpha\in\{0000,1111,0011,1100\}$. If $\alpha=0000$, then $T_{i,0}$ holds because the only row of $U_0(i)$ coincides with $M_{\langle s\rangle}$. If $\alpha=1111$, then $T_{i,1}$ holds because the only row of $U_1(i)$ coincides with $M_{\langle r\rangle}$. If $\alpha=0011$, then $T_{i,2}$ holds because the rows of $U_2(i)$ coincide with $M_{\langle s\rangle}$ and $M_{\langle r\rangle}$. If $\alpha=1100$, then $T_{i,3}$ holds because the rows of $U_3(i)$ coincide with $M_{\langle r\rangle}$ and $M_{\langle s\rangle}$. The proof of the claim is complete.

 \item \emph{If $a_3=1$, then $T_{3,0}$, $T_{3,1}$, $T_{3,4}$, or $T_{3,5}$ holds.} Suppose $a_3=1$ and let $u=\rho(3)$. If $u\in[k_M]$, then $T_{3,1}$ holds because $M_{\langle u\rangle}=D(M)_{\langle u\rangle}=U_1(3)$. Suppose otherwise and let $r,s\in[k_M]$ such that $M_{\langle r\rangle}$ is properly contained in $M_{\langle s\rangle}$ and $D(M)_{\langle u\rangle}=M_{\langle s\rangle}-M_{\langle r\rangle}$. Since $D(M)_{\langle u\rangle}$ has $0$'s in columns $5$ and $6$ and $1$'s in the remaining columns, necessarily $M_{\langle r\rangle}$ and $M_{\langle s\rangle}$ coincide in columns $5$ and $6$, $M_{\langle s\rangle}$ has $1$'s in columns $1$, $2$, $3$, and $4$, and $M_{\langle r\rangle}$ has $0$'s in the columns $1$, $2$, $3$, and $4$. Let $\beta_1$ and $\beta_2$ be the entries at columns $5$ and $6$ in row $M_{\langle r\rangle}$ (and hence also in row $M_{\langle s\rangle}$), respectively. If $\beta_1=\beta_2=0$, then $T_{3,1}$ holds because the only row of $U_1(3)$ coincides with $M_{\langle s\rangle}$. If $\beta_1=\beta_2=1$, then $T_{3,0}$ holds because the only row of $U_0(3)$ coincides with $M_{\langle s\rangle}$. If $\beta_1\beta_2=10$, then $T_{3,4}$ holds because the rows of $U_3(4)$ coincide with $M_{\langle s\rangle}$ and $M_{\langle r\rangle}$. If $\beta_1\beta_2=01$, then $T_{3,5}$ holds because the rows $U_5(3)$ coincide with $M_{\langle r\rangle}$ and $M_{\langle s\rangle}$. This completes the proof of the claim.
 
 \item \emph{$T_{4,0}$ or $T_{4,1}$ holds.} By construction, $a_4=0$. Let $u=\rho(4)$. If $u\in[k_M]$, then $T_{i,0}$ holds because $M_{\langle u\rangle}=D(M)_{\langle u\rangle}=U_0(4)$. Suppose otherwise and let $r,s\in[k_M]$ such that $M_{\langle r\rangle}$ is properly contained in $M_{\langle s\rangle}$ and $D(M)_{\langle u\rangle}=M_{\langle s\rangle}-M_{\langle r\rangle}$. Since $D(M)_{\langle u\rangle}$ has $1$'s in columns $2$, $4$, and $6$ and $0$'s in the remaining columns, necessarily $M_{\langle s\rangle}$ has $1$'s in columns $2$, $4$, and $6$, $M_{\langle r\rangle}$ has $0$'s in columns $2$, $4$, and $6$, and $M_{\langle r\rangle}$ and $M_{\langle s\rangle}$ coincide in columns $1$, $3$, and $5$. Let $\gamma_1$, $\gamma_2$, and $\gamma_3$ be the entries at columns $1$, $3$, and $5$ of $M_{\langle r\rangle}$ (and thus also $M_{\langle s\rangle}$), respectively. Let $\gamma=\gamma_1\gamma_2\gamma_3$. Notice that each row of $Y(\gamma)$ coincides with some row of $D(M)$ or the complement of some row of $D(M)$. Suppose, for a contradiction, that $\gamma$ is nonconstant. By Lemma~\ref{lem:Y}, $Y(\gamma)$ contains as a configuration some matrix $F'$ in $\ForbRow$ having fewer than $6$ columns; in particular, $F'$ has pairwise different rows. Hence, $D(M)$ contains as a configuration some matrix $F''$ that arises from $F'$ by complementing some (eventually zero) rows. By Theorem~\ref{thm:circR}, $F''$ represents the same configuration as some matrix in $\ForbRow$. As $F''$ has fewer than $6$ columns, we reach a contradiction with the choice of $F$. This contradiction shows that $\gamma$ is constant. If $\gamma=000$, then $T_{4,0}$ holds because $M_{\langle s\rangle}$ coincides with $U_0(4)$. If $\gamma=111$, then $T_{4,1}$ holds because $M_{\langle r\rangle}$ coincides with $U_1(4)$. This completes the proof of the claim.
\end{enumerate}

Because of the above claims, there is some senary sequence $\lambda=\lambda_1\lambda_2\lambda_3\lambda_4$ such that $4$ and $5$ may only occur in $\lambda$ at position $3$, $\lambda_4\in\{0,1\}$, and $T_{i,\lambda_i}$ holds for each $i\in[4]$. Hence, each row of $W(\lambda)$ coincides with some row of $M$. By Lemma~\ref{lem:W}, $W(\lambda)$ contains some matrix $F'$ in $\FDCircR$ as a configuration; in particular, $F'$ has pairwise different rows. Therefore, $M$ contains $F'$ as a configuration, contradicting the fact that $M$ contains no matrix in $\FDCircR^\infty$ as a configuration. This contradiction proves that assertion~\eqref{it:main3} implies assertion~\eqref{it:main1} and thus completes the proof of the theorem.\end{proof}

\subsection{Linear-time recognition algorithm for the $D$-circular property}\label{ssec:Dcirc-algo}

In this subsection, we give a linear-time recognition algorithm for the $D$-circular property. Given any matrix $M$, our algorithm outputs either a $D$-circular order of $M$ or some matrix in $\FDCircR^\infty$ contained in $M$ as a configuration.

According to Lemma~\ref{lem:D(M)}, in order to determine whether a matrix $M$ has the $D$-circular property, it suffices to decide whether $D(M)$ has the circular-ones property. However, a direct application of the recognition algorithm for the circular-ones property of Theorem~\ref{thm:booth-and-lueker} to $D(M)$ does not lead to a linear-time bound because, in general, $\size(D(M))$ is not bounded by a constant times $\size(M)$ even if $M$ has the $D$-circular property. In order to overcome this difficulty, we introduce a different operator $\Delta(M)$.

Let $M$ be a matrix. A \emph{maximal row of $M$} is a nontrivial row of $M$ which is not properly contained in any other nontrivial row of $M$. Analogously, a \emph{minimal row of $M$} is a nontrivial row of $M$ which does not contain properly any other nontrivial row $M$. We denote by $\Delta(M)$ a matrix that arises from $M$ by adding rows at the bottom as follows: for each minimal row $f$ of $M$ and each maximal row $g$ of $M$ such that $f$ is properly contained in $g$, add a row equal to the set difference $g-f$. Clearly, each row of $\Delta(M)$ is a row of $D(M)$ but the converse is not true in general. As $\Delta(M)$ arises from $M$ by adding some rows, we regard each linear order of the columns of $M$ as a linear order of the columns of $\Delta(M)$ and vice versa.

The next two lemmas point at proving that $D(M)$ has the circular-ones property if and only if $\Delta(M)$ has the circular-ones property and, consequently, Lemma~\ref{lem:D(M)} still holds if $D(M)$ is replaced by $\Delta(M)$.

\begin{lem}\label{lem:r,s,g-f} Let $M$ be a matrix and let $r$ and $s$ be two nontrivial rows of $M$. Let $f$ be a minimal row of $M$ such that $f\subseteq r$ and let $g$ be a maximal row of $M$ such that $s\subseteq g$. If $\leqC$ is a linear order of the columns of $M$ such that $r$, $s$, $f$, $g$, and $g-f$ are circular intervals of $\leqC$, then $s-r$ is also a circular interval of $\leqC$.\end{lem}
\begin{proof} Let $\leqC$ be a linear order of the columns of $M$ such that $r$, $s$, $f$, $g$, and $g-f$ are circular intervals of $\leqC$. Since $s$ and $g-f$ are circular intervals of $\leqC$ contained in the circular interval $g$ of $\leqC$ and $g$ is nontrivial, it follows that $s-f=s\cap(g-f)$ is a circular interval of $\leqC$. Therefore, as $f$ is nontrivial, $f\subseteq r$, and $r$ is a circular interval of $\leqC$, it follows that $r$ is not properly contained in $s-f$ and thus $s-r=(s-f)-r$ is a circular interval of $\leqC$.
This completes the proof of the lemma.\end{proof}

\begin{lem}\label{lem:D=Delta} Let $M$ be a matrix. If $\leqC$ is a linear order of the columns of $M$, then $\leqC$ is a circular-ones order of $D(M)$ if and only if $\leqC$ is a circular-ones order of $\Delta(M)$.\end{lem}
\begin{proof} As each row of $\Delta(M)$ is also a row of $D(M)$, the `only if' part follows immediately. In order to prove the `if' part, let $\leqC$ be a linear order of the columns of $M$ such that $\leqC$ is a circular-ones order of $\Delta(M)$. Let $u$ be a row of $D(M)$. If $u$ is a row of $M$, then $u$ is also a row of $\Delta(M)$ and, consequently, $u$ is a circular interval of $\leqC$. Suppose otherwise and let $r$ and $s$ be nontrivial rows of $M$ such that $r$ is properly contained in $s$ and $u=s-r$. As $r$ and $s$ are nontrivial rows of $M$, there is some minimal row $f$ of $M$ such that $f\subseteq r$ and some maximal row $g$ of $M$ such that $s\subseteq g$. Notice that $r$, $s$, $f$, $g$, and $g-f$ are rows of $\Delta(M)$ and thus circular intervals of $\leqC$. Hence, Lemma~\ref{lem:r,s,g-f} implies that $u$ is a circular interval of $\leqC$. As $u$ is an arbitrary row of $D(M)$, $\leqC$ is a circular-ones order of $D(M)$. The proof of the lemma is complete.\end{proof}

Our next lemma implies that if $M$ has the $D$-circular property, then each maximal row $g$ of $M$ contributes to a row $g-f$ in the definition of $\Delta(M)$ for at most two different minimal rows $f$.

\begin{lem}\label{lem:contains3} Let $M$ be a matrix having the circular-ones property. If some nontrivial row of $M$ properly contains at least three pairwise incomparable rows of $M$, then $M$ contains $\ZetaUnoAst$ or $\CoZetaCuatroAst$ as a configuration.\end{lem}
\begin{proof} Let $g$ be a nontrivial row of $M$ containing three pairwise incomparable rows $f_1$, $f_2$, and $f_3$ of $M$. As $M$ has the circular-ones property, by permuting the columns we assume, without loss of generality, that $M$ is a circular-ones matrix. As usual, we identify the rows and columns of $M$ with their row and column indices. As $g$ is nontrivial, we assume, without loss of generality, that $g=[1,\ell']_{[\ell]}$ for some $\ell'$ such that $1\leq\ell'<\ell$, where $\ell$ is the number of columns of $M$. As each of $f_1$, $f_2$, and $f_3$ is contained in $g$, $f_i=[a_i,b_i]_{[\ell]}$ where $1\leq a_i\leq b_i\leq\ell'$ for each $i\in[3]$. As $f_1$, $f_2$, and $f_3$ are pairwise incomparable, we assume, without loss of generality, that $a_1<a_2<a_3$ and, consequently, $b_1<b_2<b_3$. If $b_1<a_2$ and $b_2<a_3$, then $M_{\langle g,f_1,f_2,f_3\rangle,\langle a_1,a_2,a_3,\ell\rangle}=\ZetaUnoAst$. Thus, we assume without loss of generality, that $a_2\leq b_1$ or $a_3\leq b_2$. On the one hand, if $a_2\leq b_1$, then $M_{\langle f_1,f_2,g\rangle,\langle b_2,b_3,\ell,a_1,b_1\rangle}=\CoZetaCuatroAst$. On the other hand, if $a_3\leq b_2$, then $M_{\langle f_3,f_2,g\rangle,\langle a_2,a_1,\ell,b_3,a_3\rangle}=\CoZetaCuatroAst$. This completes the proof of the lemma.\end{proof}

A row of a matrix $M$ is \emph{extremal} if it is minimal or maximal in $M$. For each nonnegative integer $q$, we say that a matrix $M$ is \emph{$q$-sorted} if each nonextremal row of $M$ is either among the first $q$ rows of $M$ or occurs in $M$ below all the extremal rows of $M$.

\begin{lem}\label{lem:E_i} Let $q$ be a fixed nonnegative integer. There is a linear-time algorithm that, given any $q$-sorted matrix $M$ having the circular-ones property, outputs a $D$-circular order of $M$, a matrix in $\FDCircR$ contained in $M$ as a configuration, or the least positive integer $i$ such that $M_{\langle 1,2,\ldots,i\rangle}$ does not have the $D$-circular property.\end{lem}
\begin{proof} We assume, for the moment, that $M$ is a circular-ones matrix and that $M$ has neither trivial nor repeated rows. We argue at the end of the proof that this assumption is without loss of generality. As $M$ has no trivial rows, each row of $M$ has well-defined left and right endpoints as a circular interval of the canonical order of the columns of $M$. From now on, we will manipulate the rows as the ordered pairs of such endpoints. Hence, by traversing the endpoints circularly around the canonical order of the columns, we can easily determine all the maximal rows and the minimal rows of $M$ (notice that the sorting of the rows by their endpoints is possible in linear time using radix sort). Then, again by traversing the endpoints circularly around the canonical order of the columns, it is equally easy to generate one by one all the ordered pairs $(f,g)$ where $M_{\langle f\rangle}$ is a minimal row of $M$, $M_{\langle g\rangle}$ is a maximal row of $M$, and $M_{\langle f\rangle}$ is properly contained in $M_{\langle g\rangle}$, so that the time spent in this procedure at any time is at most proportional to the sum of $\size(M)$ and the number of such ordered pairs $(f,g)$ found so far. Hence, in order to keep the time bound within $O(\size(M))$, we keep for each maximal row $g$ a list of minimal rows properly contained in $g$, which is initially is empty, and each time a new such ordered pair $(f,g)$ is found, we add $f$ to the list of minimal rows properly contained in $g$. As soon as we detect some maximal row properly containing three minimal rows, we output a matrix in $\FDCircR$ contained in $M$ as a configuration (as in Lemma~\ref{lem:contains3}), and we are done. In fact, notice that the assumption that there are no repeated rows ensures that if a maximal row properly contains three minimal rows, then these three minimal rows are pairwise incomparable. Thus, we assume, without loss of generality, that the procedure finishes having computed all such ordered pairs $(f,g)$ and no matrix in $\FDCircR$ contained in $M$ has been output. Hence, each maximal row properly contains at most two minimal rows. As a consequence, the total number of such ordered pairs is bounded by $O(\size(M))$ and so is the time spent so far. Therefore, as $M$ has no trivial rows, every row $r$ of $M$ properly contains at most two minimal rows (otherwise, any maximal row $g$ of $M$ containing $r$ would properly contain at least three minimal rows, a contradiction).

Let $k$ be the number of rows of $M$. For each $i\in[k]$, we denote by $D_i(M)$ and $\Delta_i(M)$ the matrices $D(M_{\langle 1,2,\ldots,i\rangle})$ and $\Delta(M_{\langle 1,2,\ldots,i\rangle})$, respectively. We define a matrix $E_i(M)$ for each $i\in[k]$ as follows. Let $E_1(M)=M_{\langle 1\rangle}$ and, for each $i\in[k]$ such that $i>1$, we define $E_i(M)$ as the matrix that arises from $E_{i-1}(M)$ by applying all the following rules:
\begin{enumerate}[(i)]
 \item\label{it:Erule1} add a row equal to $M_{\langle i\rangle}$;
 
 \item\label{it:Erule2} add a row $M_{\langle i\rangle}-M_{\langle j\rangle}$ for each row $M_{\langle j\rangle}$ properly contained in $M_{\langle i\rangle}$ among those $j\in[i-1]$ such that the following two conditions hold:
 \begin{enumerate}
   \item $i\leq q$ or $M_{\langle i\rangle}$ is a maximal row of $M$, and
   \item $j\leq q$ or $M_{\langle j\rangle}$ is a minimal row of $M$;
 \end{enumerate}

 \item\label{it:Erule3} add a row $M_{\langle j\rangle}-M_{\langle i\rangle}$ for each row $M_{\langle j\rangle}$ that properly contains $M_{\langle i\rangle}$ among those $j\in[i-1]$ such that the following two conditions hold:\begin{enumerate}
  \item $j\leq q$ or $M_{\langle j\rangle}$ is a maximal row of $M$, and
  \item $i\leq q$ or $M_{\langle i\rangle}$ is a minimal row of $M$.
 \end{enumerate}
\end{enumerate}

We will show that it is possible to compute $E_k(M)$ in linear time. We begin with $E_1(M)=M_{\langle 1\rangle}$ and, for each positive integer $i$ from $1$ to $k$, we apply the above rules~\eqref{it:Erule1}, \eqref{it:Erule2}, and \eqref{it:Erule3} to convert $E_{i-1}(M)$ into $E_i(M)$. As each row of $M$ properly contains at most two minimal rows of $M$, it follows that rule~\eqref{it:Erule2} creates at most $q+2$ rows for each $i\in[k]$
and rule~\eqref{it:Erule3} creates at most $q+2$ rows for each value of $j\in[k]$. 
Hence, the number of rows of $E_k(M)$ is at most $(2q+5)k$ and the number of ones in $E_k(M)$ is at most $2q+5$ times the number of ones in $M$. As $q$ is fixed, this means that $\size(E_k(M))\in O(\size(M))$. Notice also that the time to apply the rules is also $O(\size(M))$. Indeed, on the one hand, recall that we have already determined all the pairs $(i,j)$ where $i>q$ and $j>q$ for which some of the rules~\eqref{it:Erule2} and \eqref{it:Erule3} creates a row. On the other hand, the number of pairs $(i,j)$ where $i\in[q]$ or $j\in[q]$ for which it is necessary to decide if $M_{\langle i\rangle}$ properly contains or is properly contained in $M_{\langle j\rangle}$ is at most $O(q^2+2qk)=O(k)\in O(\size(M))$ (recall that $q$ is constant).

We claim that each row of $\Delta_i(M)$ is a row of $E_i(M)$ for each $i\in[k]$. Because $E_1(M)=M_{\langle 1\rangle}$ and by virtue of rule~\eqref{it:Erule1}, row $M_{\langle j\rangle}$ is a row of $E_i(M)$ for each $j\in[i]$. Let $f',g'\in[i]$ such that $f'$ is a minimal row of $M_{\langle 1,2,\ldots,i\rangle}$, $g'$ is a maximal row of $M_{\langle 1,2,\ldots,i\rangle}$, and $M_{\langle f'\rangle}$ is properly contained in $M_{\langle g'\rangle}$. We must prove that $M_{\langle g'\rangle}-M_{\langle f'\rangle}$ is a row of $E_i(M)$. Notice that since $M$ is $q$-sorted, necessarily $f'\leq q$ or $M_{\langle f'\rangle}$ is a minimal row of $M$. Analogously, $g'\leq q$ or $M_{\langle g'\rangle}$ is a maximal row of $M$. As $f',g'\in[i]$, rules~\eqref{it:Erule2} and~\eqref{it:Erule3} ensure that $M_{\langle g'\rangle}-M_{\langle f'\rangle}$ is a row of $E_i(M)$.

As proved in the preceding paragraph, each row of $\Delta_i(M)$ is a row of $E_i(M)$. Moreover, by construction, each row of $E_i(M)$ is a row of $D_i(M)$. Hence, by Lemmas~\ref{lem:D(M)} and~\ref{lem:D=Delta}, $M_{\langle 1,2,\ldots,i\rangle}$ has the $D$-circular property if and only if $E_i(M)$ has the circular-ones property. In particular, $M$ has the $D$-circular property if and only if $E_k(M)$ has the circular-ones property. Furthermore, by the proof of Lemma~\ref{lem:D(M)} and by Lemma~\ref{lem:D=Delta}, a linear order $\leqC$ of the columns of $M$ is a $D$-circular order of $M$ if and only if $\leqC$ is a circular-ones order of $E_i(M)$ .

We apply the algorithm of Theorem~\ref{thm:booth-and-lueker} for the circular-ones property to $E_k(M)$. As $\size(E_k(M))\in O(\size(M))$, this takes linear time. If the output is a circular-ones order $\leqC$ of $E_k(M)$, we output $\leqC$ as a $D$-circular order of $M$. Thus, we assume, without loss of generality, that the output is the least positive $j$ such that $E_k(M)_{\langle 1,2,\ldots,j\rangle}$ does not have the circular-ones property. If row $j$ of $E_k(M)$ is created when adding rows to $E_{i-1}(M)$ in order to obtain $E_i(M)$, this means that $E_i(M)$ does not have the circular-ones property but each of $E_1(M)$, $E_2(M)$, $\ldots$, $E_{i-1}(M)$ has the circular-ones property. Hence, we output $i$ because $i$ is the least positive integer such that $M_{\langle 1,2,\ldots,i\rangle}$ has the $D$-circular property. The linear-time bound of the whole algorithm follows from the linear-time bound of each of its parts.

It only remains to argue that the assumption that $M$ is a circular-ones matrix and has neither trivial nor repeated rows is without loss of generality. As $M$ has the circular-ones property, by virtue of the Theorem~\ref{thm:booth-and-lueker}, it possible to compute in linear time a circular-ones order $\leqC$ of $M$. We represent each nontrivial row of $M$ as the ordered pair of its (well-defined) left and right endpoints as a circular interval of $\leqC$. With this representation, we can detect all repetitions of rows also within linear time (e.g., by sorting with radix sort). In this way, we identify for each set of repeated rows, the topmost occurrence among them in $M$. We let $M'$ be the matrix that arises from $M$ by permuting the columns so that the canonical order of the columns coincides with $\leqC$ and by removing all the trivial rows and, for each set of repeated rows, removing all the occurrences of these rows except for the topmost. As $M$ is $q$-sorted, $M'$ is also $q$-sorted. Moreover, $M'$ satisfies the assumption at the beginning of the proof; i.e., $M$ is a circular-ones matrix and has neither trivial nor repeated rows. Hence, we can apply the algorithm given above with $M'$ playing the role of $M$ in order to either determine a $D$-circular order $\leqCprime$ of $M'$, a matrix $F$ in $\FDCircR$ contained in $M'$ as a configuration, or the least positive integer $i'$ such that $M'_{\langle 1,2,\ldots,i'\rangle}$ does not have the $D$-circular property. From a $D$-circular order $\leqCprime$ of $M'$, we can compute a $D$-circular order of $M$ by reverting the permutation of the columns performed on $M$ in order to obtain $M'$. That this indeed leads to a $D$-circular order of $M$ follows from the fact that $D$-circularity is preserved by the addition of trivial and already present rows. Given a matrix $F$ in $\FDCircR$ contained in $M'$ as a configuration, we can find $F$ contained in $M$ as a configuration, once more by taking into account the permutation applied to the columns of $M$ in order to obtain $M'$ and the correspondence between the rows of $M'$ and the rows of $M$. Finally, if we are given the least positive integer $i'$ such that $M'_{\langle 1,2,\ldots,i'\rangle}$ does not have the $D$-circular property, then we output the index $i$ in $M$ corresponding to row $i'$ of $M'$. This output is correct because, as $D$-circularity is invariant by the addition of trivial and already present rows, then no index $j$ of $M$ corresponding to a trivial row of $M$ or to an occurrence of a nontrivial row of $M$ different from the topmost, can be the least positive integer such that $M_{\langle 1,2,\ldots,j\rangle}$ does not have the $D$-circular property. All these task can easily be implemented to take linear time.
This completes the proof of the lemma.\end{proof}

We say that a matrix property $\mathcal P$ is \emph{hereditary} if, for each matrix $M$ having property $\mathcal P$ and each matrix $M'$ contained in $M$ as a configuration, $M'$ has property $\mathcal P$. A matrix $M$ is a \emph{minimal forbidden submatrix for property $\mathcal P$} if $M$ is the only submatrix of $M$ not having property $\mathcal P$. Let $M$ be a $k\times\ell$ matrix. We say that a matrix \emph{$M'$ arises from $M$ by a cut-and-antishift operation} if there exists some $k'\in[k]$ such that $M'=M_{\langle k',1,2,\ldots,k'-1\rangle}$. If $q$ is a nonnegative integer and $\mathcal N$ is a matrix class, we denote by $\mathcal S^q[\mathcal N]$ the class of matrices that arise from matrices in $\mathcal N$ by applying a sequence of at most $q$ cut-and-antishift operations.

In the lemma below, we adapt to our setting the strategy underlying the algorithm \texttt{TuckerRows} in Section~4.1 of~\cite{MR3447121}. Indeed, the results in that section correspond to the case of the result below where $\mathcal P$ is the consecutive-ones property, $\mathcal M$ and $\mathcal N$ are equal to the class of all matrices, and algorithm $\Pi$ is the algorithm of Theorem~\ref{thm:booth-and-lueker} for the consecutive-ones property.

\begin{lem}[adapted from Lemma~4.2 of \cite{MR3447121}]\label{lem:algo_iterated} Let $\mathcal P$ be a hereditary matrix property and let $\mathcal M$ be a matrix class. Suppose that there is a linear-time algorithm $\Pi$ that, given any matrix $M$ in $\mathcal M$ such that $M$ does not have property $\mathcal P$, outputs either a minimal forbidden submatrix for property $\mathcal P$ contained in $M$ as configuration or the least positive integer $i$ such that $M_{\langle 1,2,\ldots,i\rangle}$ does not have property $\mathcal P$. For each fixed nonnegative integer $q$ and each matrix class $\mathcal N$ such that $\mathcal S^q[\mathcal N]\subseteq\mathcal M$, there is a linear-time algorithm that, given any matrix $M$ in $\mathcal N$ such that each minimal forbidden submatrix for property $\mathcal P$ contained in $M$ as a configuration has at most $q$ rows, outputs a matrix contained in $M$ as a configuration having at most $q$ rows and not having property $\mathcal P$.\end{lem}
\begin{proof} Let $M$ be a matrix in $\mathcal N$ such that $M$ does not have property $\mathcal P$ and such that every minimal forbidden submatrix for property $\mathcal P$ contained in $M$ as a configuration has at most $q$ rows. We begin by letting $M'=M$. We repeatedly perform an iteration which consists in applying algorithm $\Pi$ to $M'$ plus doing the following: if the output of algorithm $\Pi$ is a minimal forbidden submatrix $F$ for property $\mathcal P$ contained in $M'$ as a configuration, then output $F$ and stop; if, on the contrary, the output is the least positive integer $i$ such that $M'_{\langle 1,2,\ldots,i\rangle}$ does not have property $\mathcal P$, then we replace $M'$ by $M'_{\langle i,1,2,\ldots,i-1\rangle}$. We repeat this iteration until some such matrix $F$ is output or $q+1$ iterations were completed. Since $\mathcal S^q[\mathcal N]\subseteq\mathcal M$ and $\Pi$ is linear-time, each iteration can be performed in linear time. Moreover, as $q$ is fixed, all the iterations can be completed in linear time. Suppose first that at some iteration some matrix $F$ is output. As $M'$ is contained in $M$ as a configuration, $F$ is also contained in $M$ as a configuration. Thus, since each minimal forbidden submatrix for property $\mathcal P$ contained in $M$ as a configuration has at most $q$ rows, $F$ has at most $q$ rows and thus the output is correct. Hence, we assume, without loss of generality, that the total $q+1$ iterations are completed and no matrix $F$ has been output. By construction, after the completion of each iteration, $M'$ does not have property $\mathcal P$ and the removal of row $1$ from $M'$ leaves a matrix having property $\mathcal P$. Let $M'$ denote the value of $M'$ after the $q+1$ iterations.  Suppose, for a contradiction, that $M'$ has at least $q+1$ rows. Thus, by induction and the nature of the cut-and-antishift operation, $M'$ does not have property $\mathcal P$ and the removal of any of the first $q+1$ rows from $M'$ would produce a matrix having property $\mathcal P$. Hence, the first $q+1$ rows of $M'$ are part of every minimal forbidden submatrix for property $\mathcal P$ contained in $M'$ as a configuration. Since $M'$ does not have property $\mathcal P$, $M'$ contains as a configuration some minimal forbidden submatrix for property $\mathcal P$ having at least $q+1$ rows. This contradicts the fact that $M$ contains as a configuration $M'$ but no minimal forbidden submatrix for property $\mathcal P$ having more than $q$ rows. This contradiction shows that $M'$ has at most $q$ rows. By construction, $M'$ does not have property $\mathcal P$ and so $M'$ is a correct as output. This completes the proof of the lemma.\end{proof}

We are now ready to prove the main algorithmic result of this work.

\begin{thm}\label{thm:algo_Dcirc} There is a linear-time algorithm that, given any matrix $M$, outputs either a $D$-circular order of $M$ or a matrix in $\FDCircR^\infty$ contained in $M$ as a configuration.\end{thm}
\begin{proof} We apply the linear-time recognition algorithm of Theorem~\ref{thm:booth-and-lueker} for the circular-ones property to $M$. If it turns out that $M$ does not have the circular-ones property, we apply the algorithm of Theorem~\ref{thm:circR} in order to obtain in linear time some matrix $F$ in $\ForbRow$ contained in $M$ as a configuration. Then, if follows from the proof of Lemma~\ref{lem:forb_Circ1R} that it is possible to find a matrix in $\FDCircR^\infty$ contained in $M$ as a configuration in linear time and we are done. Thus, we assume, without loss of generality, that the algorithm of Theorem~\ref{thm:booth-and-lueker} detects that $M$ has the circular-ones property. We apply the algorithm of Lemma~\ref{lem:E_i} to $M$. If the output is a $D$-circular order of $M$ or a matrix in $\FDCircR$ contained in $M$ as a configuration, we are done. Hence, we assume, without loss of generality, that $M$ does not have the $D$-circular property. By Theorem~\ref{thm:main}, the minimal forbidden submatrices for the $D$-circular property are those in $\FDCircR^\infty$. However, as $M$ has the circular-ones property, $M$ contains no $\MIast k$ and no $\overline{\MIast k}$ as a configuration for any $k\geq 3$. Therefore, each minimal forbidden submatrix for the $D$-circular property contained in $M$ as a configuration belongs to $\FDCircR$ and, in particular, has at most $4$ rows. We permute the rows of $M$ so that all the extremal rows are above any nonextremal row; i.e., so that $M$ is $0$-sorted. If we denote by $\mathcal M_q$ the class of $q$-sorted matrices, then $M\in\mathcal M_0$. Moreover, clearly, $\mathcal S^q[\mathcal M_0]\subseteq\mathcal M_q$. If we let $q=4$, $\mathcal P$ be the $D$-circular property, $\mathcal M=\mathcal M_4$, $\mathcal N=\mathcal M_0$, and $\Pi$ be the algorithm of Lemma~\ref{lem:E_i}, then Lemma~\ref{lem:algo_iterated} ensures that in linear time it is possible to find a matrix $M'$ contained in $M$ as a configuration such that $M'$ has at most $4$ rows and $M'$ does not have the $D$-circular property. Since $M$ has the circular-ones property and $M'$ is contained in $M$ as a configuration, also $M'$ has the circular-ones property. Hence, by virtue of Theorem~\ref{thm:main}, $M'$ contains some matrix in $\FDCircR$ as a configuration. As $M'$ has at most $4$ rows and each of the matrices in $\FDCircR$ has at most $5$ columns, some matrix $F$ in $\FDCircR$ contained in $M'$ as a configuration can be easily found in linear time (notice that because of having at most $4$ rows, $M'$ has at most $16$ pairwise different columns). As $M'$ is contained in $M$ as a configuration, $F$ is contained in $M$ as a configuration. The linear-time bound for the whole algorithm follows from the linear-time bounds in Theorems~\ref{thm:booth-and-lueker} and \ref{thm:circR} and Lemmas~\ref{lem:E_i} and~\ref{lem:algo_iterated}. The proof of the theorem is complete.\end{proof}

\subsection{Connection with the $D$-interval property}\label{ssec:Dint}

We now show that the results obtained for the $D$-circular property in the preceding subsections generalize some results in the literature regarding the $D$-interval property. Namely, we will show how a linear-time recognition algorithm and the minimal forbidden submatrix characterization can be derived from our above results. The corresponding set of minimal forbidden submatrices is
\[ \FDIntR^\infty=\FDIntR\cup\{Z_1\trans\}\cup\{\MI k\colon\,k\geq 3\} \]
where
\[ \FDIntR=\{Z_1,Z_2,Z_3,Z_2\trans,Z_3\trans\}. \]
(See Figures~\ref{fig:TuckerMatrices} and~\ref{fig:smallmatrices}.)

The first linear-time algorithms for recognizing the $D$-interval property and producing a corresponding $D$-interval order whenever possible were given in~\cite{MR944699} and~\cite{MR1369371}. A linear-time recognition algorithm which, in addition, outputs a matrix in $\FDIntR^\infty$ contained in $M$ as a configuration whenever $M$ does not have the $D$-interval property was given in~\cite{MR2134416}. We derive a recognition algorithm also with this ability from our Theorem~\ref{thm:algo_Dcirc}.

\begin{thm}[\cite{MR917130,MR2134416}]\label{thm:algo_Dint} There is a linear-time algorithm that, given any matrix $M$, outputs either a $D$-interval order of $M$ or a matrix in $\FDIntR^\infty$ contained in $M$ as a configuration.\end{thm}
\begin{proof} We apply the algorithm of Theorem~\ref{thm:algo_Dcirc} to $M$. If the output is a $D$-circular order $\leqCast$ of $M^*$, then we output the restriction $\leqC$ of $\leqCast$ to the columns of $M$, which can be easily seen to be a $D$-interval order of $M$. Thus, we assume, that the output of the algorithm of Theorem~\ref{thm:algo_Dcirc} applied to $M^*$ is some matrix $F$ in $\FDCircR^\infty$ contained in $M^*$ as a configuration. Some submatrix $F'$ in $\FDIntR^\infty$ contained in $F$ as a configuration can be found in $O(1)$ time (in fact, notice that $\MIast k_{\id_k,\id_k}=\MI k$ and $\overline{\MIast k}_{\langle 1,2,3,k\rangle,\langle k+1,1,2\rangle}=Z_3$ for each $k\geq 4$). Hence, if $F$ is contained as a configuration in $M^*$ without involving the last column of $M^*$, then a matrix $F'$ in $\FDIntR^\infty$ contained in $M$ as a configuration can be found in linear time. Therefore, we assume, without loss of generality, that $F$ is contained in $M^*$ as a configuration in such a way that some column of $F$ corresponds to the last column of $M^*$. In particular, $F$ has some empty column. As $F\in\FDCircR^\infty$, $F$ must be $\ZetaUnoAst$, $\ZetaDosAst$, $\ZetaTresAst$, $\ZetaCuatroAst$, $\CoZetaCuatroAst$, or $\MIast k$ for some $k\geq 3$. Hence, if $F'$ is the matrix that arises from $F$ by removing its empty column, then $F'$ represents the same configuration as $Z_1$, $Z_2$, $Z_3$, $Z_2\trans$, $Z_3\trans$, or $\MI k$ (because $Z_4=Z_2\trans$ and $\CoZetaCuatroAst$ represents the same configuration as $(Z_3\trans)^*$) and thus we can find $F'$ contained in $M$ as a configuration in linear time. In this way, it is possible to find some matrix in $\FDIntR^\infty$ contained in $M$ in linear time. This completes the proof of the theorem.\end{proof}

The above result immediately implies the following structural characterizations of the $D$-interval property, including the characterization by minimal forbidden submatrices due to Moore~\cite{MR0437403}. Alternative proofs of the result below were subsequently given in different contexts; e.g., it also follows by combining results from~\cite{MR0221974} and \cite{MR917130} or, alternatively, from \cite{MR0295938} and \cite{MR1440521}.

\begin{thm}[\cite{MR0437403}]\label{thm:Dint} For each matrix $M$, the following conditions are equivalent:
 \begin{enumerate}[(i)]
  \item\label{it1:Dint} $M$ has the $D$-interval property;
  \item\label{it2:Dint} $M\trans$ has the $D$-interval property;
  \item\label{it3:Dint} $M$ has the consecutive-ones property for rows and contains no matrix in $\FDIntR$ as a configuration;
  \item\label{it4:Dint} $M$ contains no matrix in $\FDIntR^\infty$ as a configuration.
\end{enumerate}\end{thm}
\begin{proof} Assertion~\eqref{it1:Dint} implies assertion~\eqref{it4:Dint} because none of the matrices in $\FDIntR^\infty$ has the $D$-interval property and every matrix contained as a configuration in a matrix having the $D$-interval property has the $D$-interval property. Conversely, Theorem~\ref{thm:algo_Dint} shows that assertion~\eqref{it4:Dint} implies assertion~\eqref{it1:Dint}. Thus, assertions~\eqref{it1:Dint} and~\eqref{it4:Dint} are equivalent. The equivalence extends to assertion~\eqref{it2:Dint} because the transpose of each matrix in $\FDIntR^\infty$ represents the same configuration as some matrix in $\FDIntR^\infty$. Clearly, assertion~\eqref{it1:Dint} implies assertion~\eqref{it3:Dint}. Moreover, assertion~\eqref{it3:Dint} implies assertion~\eqref{it4:Dint} because clearly neither $Z_1\trans$ nor $\MI k$ for any $k\geq 3$ has the consecutive-ones property for rows. This completes the proof the equivalence among assertions~\eqref{it1:Dint}, \eqref{it2:Dint}, \eqref{it3:Dint}, and \eqref{it4:Dint} and thus of the theorem.\end{proof}

\section{Linearly and circularly compatible ones properties}\label{sec:LCO+CCO}

The main results of this section are a minimal forbidden submatrix characterization and a linear-time recognition algorithm for the circularly compatible ones property.
The corresponding set of minimal forbidden submatrices is
\begin{align*}
  \FCCO^\infty=\FCCO\cup\bigcup_{k=3}^\infty\{\MIast k,\overline{\MIast k},\MIast k\trans,\overline{\MIast k}\trans\},
\end{align*}
where
\[   \FCCO=\{\ZetaDosAst,\ZetaTresAst,\ZetaCuatroAst,\ZetaCinco,\CoZetaDosAst,\CoZetaCuatroAst,
            (\ZetaDosAst)\trans,(\ZetaTresAst)\trans,(\ZetaCuatroAst)\trans,\ZetaCincoTrans,(\CoZetaDosAst)\trans,(\CoZetaCuatroAst)\trans\}. \]
(See Figure~\ref{fig:forb_DCircRinfty}.)

This section is organized as follows. In Subsection~\ref{ssec:LCO}, we argue that the linearly compatible ones property coincides with the $D$-interval property. In Subsection~\ref{ssec:CCO}, we give a minimal forbidden submatrix characterization and a linear-time recognition algorithm for the circularly compatible ones property.

\subsection{Linearly compatible ones property}\label{ssec:LCO}

In this subsection, we will briefly argue that the linearly compatible ones property coincides with the $D$-interval property. 

Let $M$ be a matrix and let $(\leqR,\leqC)$ be some biorder of $M$. Let $r_1,r_2,\ldots,r_p$ be all the nonempty rows of $M$ in ascending order of $\leqR$. The biorder $(\leqR,\leqC)$ is a \emph{monotone consecutive biorder of $M$}~\cite{MR1271987} or \emph{forward-convex biorder of $M$}~\cite{MR1440521} if  $\leqC$ is a consecutive-ones order of $M$ and, if $r_i$ equals the linear interval $[d_i,e_i]_{\leqC}$ for each $i\in[p]$, then $d_1d_2\ldots d_p$ and $e_1e_2\ldots e_p$ are monotone on $\leqC$. A matrix has the \emph{monotone consecutive property} or the \emph{forward-convex property} if it admits some monotone consecutive biorder (or, equivalently, a forward-convex biorder). Notice that in the definition of the monotone consecutive property $r_1,r_2,\ldots,r_p$ are all the \emph{nonempty} rows, whereas in the definition of the linearly compatible ones property $r_1,r_2,\ldots,r_p$ are all the \emph{nontrivial} rows only. The biorder $(\leqR,\leqC)$ is a \emph{doubly forward-convex biorder of $M$} if $(\leqR,\leqC)$ is a forward-convex biorder of $M$ and $(\leqC,\leqR)$ is a forward-convex biorder of $M\trans$. A matrix has the \emph{doubly forward-convex property} if it admits some doubly forward-convex biorder.

\begin{thm}[\cite{MR1440521,MR1271987,MR917130,MR1368750}]\label{thm:MCA} If $M$ is a matrix, then the following assertions are equivalent:
\begin{enumerate}[(i)]
 \item\label{it1:MCA} $M$ has the $D$-interval property;
 \item\label{it2:MCA} $M$ has the monotone consecutive property (or, equivalently, the forward-convex property);
 \item\label{it3:MCA} $M$ has the doubly forward-convex property.
\end{enumerate}\end{thm}

We observe that by combining Theorems~\ref{thm:Dint} and~\ref{thm:MCA}, it follows that the linearly compatible ones property coincides with the $D$-interval property.

\begin{cor}\label{cor:LCO} A matrix $M$ has the linearly compatible ones property if and only if $M$ has the $D$-interval property.\end{cor}
\begin{proof} As none of the matrices in $\FDIntR^\infty$ has the linearly compatible ones property, Theorem~\ref{thm:Dint} implies that if $M$ has the linearly compatible ones property, then $M$ has the $D$-interval property. Conversely, if $M$ has the $D$-interval property, then Theorem~\ref{thm:MCA} ensures that $M$ has the doubly forward-convex property and, in particular, $M$ has the linearly compatible ones property. This completes the proof of the corollary.\end{proof}

Sen and Sanyal~\cite{MR1271987} characterized proper interval bigraphs as the bipartite graphs associated with matrices having the monotone consecutive property. Therefore, because of the equivalence between the linearly compatible ones property and the monotone consecutive property, their result implies Theorem~\ref{thm:pib_cco} in the introduction.

\subsection{Circularly compatible ones property}\label{ssec:CCO}

In this subsection, by combining the results in the preceding section for the $D$-circular property with some results from~\cite{MR3065109}, we derive a minimal forbidden submatrix characterization and a linear-time recognition for the circularly compatible ones property.

Basu et al.~\cite{MR3065109} proved that, for matrices having no trivial rows, the $D$-circular property coincides with the \emph{monotone circular property} defined as follows. Let $\leqX$ be a linear order on some set finite set $X$. Let $X=\{x_1,x_2,\ldots,x_k\}$ where $x_1\lessX x_2\lessX\cdots\lessX x_k$. We denote by $\leqXX$ the linear order on the set $\XX=\{x_1,x_2,\ldots,x_k,x_1^+,x_2^+,\ldots,x_k^+\}$ such that $x_1\lessXX x_2\lessXX\cdots\lessXX x_k\lessXX x_1^+\lessXX x_2^+\lessXX\cdots\lessXX x_k^+$. With each nontrivial circular interval $[a,b]_{\leqX}$, we associate its \emph{unwrapped interval relative to $\leqX$}, denoted $\intXX{a,b}$, which is the linear interval $[a,c]_{\leqXX}$ where $c=b$ if $a\leqX b$, and $c=b^+$ if $b\lessX a$. Let $M$ be a matrix and let $(\leqR,\leqC)$ be a biorder of $M$. Let $r_1,r_2,\ldots,r_p$ be all the nontrivial rows of $M$ in ascending order of $\leqR$. We say $M$ has \emph{monotone left endpoints with respect to $(\leqR,\leqC)$} if $\leqC$ is a circular-ones order of $M$ and if $r_i=[d_i,e_i]_{\leqC}$ for each $i\in[p]$, then $d_1d_2\ldots d_p$ is monotone on $\leqC$. We say that $M$ has \emph{circularly monotone unwrapped right endpoints} if $\leqC$ is a circular-ones order of $M$ and if $r_i=[d_i,e_i]_{\leqC}$ for each $i\in[p]$ and $[d_i,f_i]_{\leqCC}$ is the unwrapped interval of $[d_i,e_i]_{\leqC}$ with respect to $\leqC$ (i.e., $f_i=e_i$ if $d_i\leqC e_i$, whereas $f_i=e_i^+$ if $e_i\lessC d_i$), then $f_1f_2\ldots f_p$ is monotone on $\leqCC$. If $M$ has no trivial rows, a \emph{monotone circular biorder}~\cite{MR3065109} of $M$ is a biorder $(\leqR,\leqC)$ of $M$ such that all the following conditions hold:
\begin{enumerate}[(i)]
 \item $M$ has monotone left endpoints $e_1e_2\ldots e_p$ with respect to $(\leqR,\leqC)$;
 \item $M$ has monotone unwrapped right endpoints $f_1f_2\ldots f_p$ with respect to $(\leqR,\leqC)$;
 \item\label{it3:MCircAdef} either $f_1=e_1^+$ or both $f_1=e_1$ and $f_p\leqCC e_1^+$.
\end{enumerate}
We call condition \eqref{it3:MCircAdef} above the \emph{alignment condition}.\footnote{In the notation of~\cite{MR3065109}, the alignment condition is equivalent to $\mu_m\leq\mu_1+n$. Although this condition is not part of Definition~3.4 in~\cite{MR3065109}, it is implicitly assumed throughout~\cite{MR3065109}. For instance, $M=\left(\begin{smallmatrix}
      1 & 0 & 0 & 0\\
      1 & 1 & 1 & 0\\
      0 & 0 & 1 & 1\\
      1 & 1 & 0 & 1
\end{smallmatrix}\right)$ has no trivial rows and no trivial columns and its canonical biorder $(\leqR,\leqC)$ satisfies the definition of monotone circular biorder of $M$ given here, except for the alignment condition. However, $M$ does not have the following properties asserted for such matrices in~\cite[p.\ 12 and Theorem~3.4]{MR3065109}: (a) all the columns of $M$ are circular intervals of $\leqR$, (b) $M$ has circularly monotone right endpoints with respect to $(\leqR,\leqC)$, and (c) the bipartite graph associated with $M$ is a proper circular-arc bigraph. This is, nevertheless, a minor omission because properties (a), (b), and (c) are fulfilled once Definition~3.4 of~\cite{MR3065109} is amended by adding the alignment condition.} A matrix $M$ having no trivial rows has the \emph{monotone circular property}~\cite{MR3065109} if it admits a monotone circular biorder. The aforementioned result in~\cite{MR3065109} relating the $D$-circular property with the monotone circular property is the following.

\begin{thm}[\cite{MR3065109}]\label{thm:Basu} If $M$ is a matrix having no trivial rows, then the following assertions are equivalent:
\begin{enumerate}[(i)]
 \item\label{it1:Basu} $M$ has the $D$-circular property;
 \item\label{it2:Basu} $M$ has the monotone circular property.
\end{enumerate}\end{thm}

Let $M$ be a matrix and let $(\leqR,\leqC)$ be a biorder of $M$. We say $M$ has \emph{circularly monotone right endpoints with respect to $(\leqR,\leqC)$} if $\leqC$ is a circular-ones order of $M$ and if $r_i=[d_i,e_i]_{\leqC}$ for each $i\in[p]$, then $e_1e_2\ldots e_p$ is circularly monotone on $\leqC$. The following was observed in~\cite[p.\ 12]{MR3065109}.

\begin{lem}[\cite{MR3065109}]\label{lem:Basu} If $M$ is a matrix having no trivial rows and $(\leqR,\leqC)$ is a monotone circular biorder of $M$, then the columns of $M$ are circular intervals of $\leqR$ and $M$ has circularly monotone right endpoints with respect to $(\leqR,\leqC)$.\end{lem}

Hence, if a matrix $M$ has no trivial rows and admits some monotone circular biorder, then $M$ has the circularly compatible ones property. This fact combined with Theorem~\ref{thm:Basu} implies the following result.

\begin{cor}\label{cor:CCO} If a matrix $M$ has no trivial rows and has the $D$-circular property, then $M$ has the circularly compatible ones property.\end{cor}

For arbitrary matrices (i.e., with trivial rows allowed), the $D$-circular property is not always sufficient to ensure the circularly compatible ones property. We will show that if trivial rows are allowed, then the circularly compatible ones property is equivalent to the doubly $D$-circular property defined as follows: a matrix $M$ has the \emph{doubly $D$-circular property} if $M$ and $M\trans$ have the $D$-circular property. For that purpose, we will introduce the notation $M^{[u]}$ and derive the lemma below from our Theorem~\ref{thm:main}.

If $M$ is a matrix having some trivial row consisting of entries all equal to $u$ for some $u\in\{0,1\}$, we denote by $M^{[u]}$ the matrix that arises from $M$ by adding a last column having a $1-u$ entry at each row where $M$ has all entries equal to $u$, and having a $u$ entry in all the remaining rows. By construction, $M^{[u]}$ has no trivial rows.

\begin{lem}\label{lem:M[u]} Let $M$ be a matrix having some trivial row consisting of entries all equal to $u$ for some $u\in\{0,1\}$. If $M$ has the doubly $D$-circular property, then $M^{[u]}$ has the doubly $D$-circular property.\end{lem}
\begin{proof} Let $k$ and $\ell$ be the number of rows and columns of $M$, respectively.

We consider first the case $u=0$. Suppose, for a contradiction, that $M$ has the doubly $D$-circular property but $M^{[0]}$ does not have the doubly $D$-circular property. By Theorem~\ref{thm:main}, $M^{[0]}$ contains as a configuration some matrix $F$ in $\FDCircR^\infty\cup(\FDCircR^\infty)\trans$ where $(\FDCircR^\infty)\trans=\{F\trans\colon\,F\in\FDCircR^\infty\}$. Let $\rho$ and $\sigma$ be a row and a column map such that $(M^{[0]})_{\rho,\sigma}=F$. As $M$ has the doubly $D$-circular property, $M$ does not contain $F$ as a configuration. Hence, necessarily $\ell+1$ (i.e., the index of the last column of $M^{[0]}$) belongs to the image of $\sigma$; let $j=\sigma^{-1}(\ell+1)$. Clearly, each matrix in $\FDCircR^\infty$ has pairwise different rows and pairwise different columns and it can be verified by inspection that no matrix in $\FDCircR^\infty$ has an entry equal to $1$ such that every other entry in the same column and every other entry in the same row is equal to $0$. Thus, for each $i$ in the image of $\rho$, row $i$ of $M$ has some entry different from $0$. Therefore, by construction, all the entries in column $j$ of $F$ are equal to $0$ and $M$ contains as a configuration the matrix $F'$ that arises from $F$ by removing column $j$ and adding one row with all entries equal to $0$. By inspection, the ordered pair $(F,j)$ must be $(\ZetaUnoAst,4)$, $(\ZetaDosAst,4)$, $(\ZetaTresAst,4)$, $(\ZetaCuatroAst,5)$, or $(\CoZetaCuatroAst,3)$, and the corresponding matrix $F'$ contains $\overline{\MIast 3}\trans$, $(\ZetaCuatroAst)\trans$, $\CoZetaCuatroAst\trans$, $(\ZetaDosAst)\trans$, or $(\ZetaTresAst)\trans$, respectively, as a configuration. As $M$ contains $F'$ as a configuration, $M$ does not have the doubly $D$-circular property. This contradiction proves the lemma in case $u=0$.

The proof for the case where $u=1$ follows from the above proof for the case where $u=0$ because $M^{[1]}=\overline{\overline M^{[0]}}$, the doubly $D$-circular property is invariant under matrix complementation, and if $M$ has a row having all entries equal to $1$ then $\overline M$ has a row having all entries equal to $0$. This completes the proof of the lemma.\end{proof}

We need one more lemma.

\begin{lem}\label{lem:CCO=>Dcirc} Each matrix in $\FDCircR^\infty$ contains some matrix in $\FCCO^\infty$ as a configuration.\end{lem}
\begin{proof} Notice that the only matrices in $\FDCircR^\infty$ but not in $\FCCO^\infty$ are, up to complementation, $\ZetaUnoAst$, $\ZetaSeis$, $\ZetaSiete$, and $\ZetaOcho$. On the one hand, each of $\ZetaUnoAst$, $\ZetaSeis$, $\ZetaSiete$ contains $\overline{\MIast 3}\trans$ as a configuration. On the other hand, $\ZetaOcho$ contains $(\ZetaDosAst)\trans$ as a configuration. Since the complement of each matrix in $\FCCO^\infty$ represents the same configuration as some matrix in $\FCCO^\infty$, the proof of the lemma is complete.\end{proof}

The following is the main structural result of this section and characterizes the circularly compatible ones property for arbitrary matrices by minimal forbidden submatrices.

\begin{thm}\label{thm:CCO} For each matrix $M$, the following assertions are equivalent:
\begin{enumerate}[(i)]
 \item\label{it1:CCO} $M$ has the circularly compatible ones property;
 \item\label{it2:CCO} $M$ contains no matrix in $\FCCO^\infty$ as a configuration;
 \item\label{it3:CCO} $M$ has the circular-ones property for rows and for columns and contains as a configuration no matrix in $\FCCO$ as a configuration;
 \item\label{it4:CCO} $M$ has the doubly $D$-circular property.
\end{enumerate}\end{thm}
\begin{proof} As none of the matrices in $\FCCO^\infty$ has the circularly compatible ones property and the circularly compatible ones property of a matrix $M$ is inherited by every matrix contained in $M$ as a configuration, assertion~\eqref{it1:CCO} implies assertion~\eqref{it2:CCO}. As each matrix in $\ForbRow$ contains some matrix in $\FDCircR^\infty$ as a configuration (by Lemma~\ref{lem:forb_Circ1R}) and each matrix in $\FDCircR^\infty$ contains some matrix in $\FCCO^\infty$ (by Lemma~\ref{lem:CCO=>Dcirc}), it follows that each matrix in $\ForbRow$ contains some matrix in $\FCCO^\infty$ as a configuration. Hence, since $\FCCO^\infty$ is closed under matrix transposition and $\FCCO\subseteq\FCCO^\infty$, Theorem~\ref{thm:circR} shows that assertion~\eqref{it2:CCO} implies assertion~\eqref{it3:CCO}. It follows from the proof of Lemma~\ref{lem:CCO=>Dcirc} that every matrix in $\FDCircR$ contains some matrix in $
\FCCO$ as a configuration. Therefore, since $\FCCO$ is closed under matrix transposition, Theorem~\ref{thm:main} shows that assertion~\eqref{it3:CCO} implies assertion~\eqref{it4:CCO}. In order to complete the proof of theorem it only remains to prove that assertion~\eqref{it4:CCO} implies assertion~\eqref{it1:CCO}. For that purpose, suppose that assertion~\eqref{it4:CCO} holds; i.e., $M$ has the doubly $D$-circular property. If $M$ has no trivial row, then Corollary~\ref{cor:CCO} ensures that $M$ has the circularly compatible ones property. Hence, we assume, without loss of generality, that $M$ has some trivial row. Thus, $M$ has some row whose entries are all equal to $u$ for some $u\in\{0,1\}$. By Lemma~\ref{lem:M[u]}, $M^{[u]}$ has the doubly $D$-circular property. Since $M^{[u]}$ has no trivial row, Corollary~\ref{cor:CCO} ensures that $M^{[u]}$ has the circularly compatible ones property. As $M$ is a submatrix of $M^{[u]}$, also $M$ has the circularly compatible ones property. This proves that assertion~\eqref{it4:CCO} implies assertion~\eqref{it1:CCO} and the proof of the theorem is complete.
\end{proof}

K{\"o}bler, Kuhnert, and Verbitsky~\cite{MR3573905} proved that a graph is a proper circular-arc graph if and only if its augmented adjacency matrix has the $D$-circular property. Because of Theorem~\ref{thm:pca_cco}, their result is equivalent to the theorem below. We observe that we can derive the result as a special case of Theorem~\ref{thm:CCO}.

\begin{thm}[\cite{TuckerPhD,MR3573905}] If $M$ is the augmented adjacency matrix of some graph, then $M$ has the circularly compatible ones property if and only if $M$ has the $D$-circular property.\end{thm}
\begin{proof} It follows from the equivalence between assertions~\eqref{it1:CCO} and~\eqref{it4:CCO} in Theorem~\ref{thm:CCO}. In fact, as $M$ is symmetric, $M$ has the $D$-circular property if and only if $M$ has the doubly $D$-circular property.\end{proof}

We will now give a linear-time recognition algorithm for the circularly compatible ones property. In order to do so, we will rely on the following result which corresponds to the proof of implication (2)${}\Rightarrow{}$(3) of Theorem~3.4 in~\cite{MR3065109}.

\begin{lem}[\cite{MR3065109}]\label{lem:D-circ=>MCIrcA} Let $M$ be a matrix having no trivial rows and having some $D$-circular order $\leqC$. Let $r_1,r_2,\ldots,r_p$ be all the rows of $M$ in ascending order of some linear order $\leqR$ such that $r_i=[d_i,e_i]_{\leqC}$ for each $i\in[p]$ and, for every two $i,j\in[p]$, if $r_i\leqR r_j$ then either $d_i\lessR d_j$ or both $d_i=d_j$ and $r_i\subseteq r_j$. Then, $(\leqR,\leqC)$ is a monotone circular biorder.\end{lem}

We are ready to give the main algorithmic result of this section.

\begin{thm}\label{thm:algo_CCO} There is a linear-time algorithm that, given any matrix $M$, outputs either a circularly compatible ones biorder of $M$ or a matrix in $\FCCO^\infty$ contained in $M$ as a configuration.\end{thm}
\begin{proof} We suppose, for the moment, that $M$ has no trivial rows. We argue at the end of this proof that this assumption is without loss of generality. We begin by applying the linear-time algorithm in Theorem~\ref{thm:algo_Dcirc}. If the output is a matrix in $\FDCircR^\infty$ contained in $M$ as a configuration, we are done because this leads to a matrix in $\FCCO^\infty$ contained in $M$ as a configuration in additional $O(1)$ time (see the proof of Lemma~\ref{lem:CCO=>Dcirc}). Thus, we assume, without loss of generality, that the output is a $D$-circular order $\leqC$ of $M$. As we are assuming that $M$ has no trivial rows, Lemma~\ref{lem:D-circ=>MCIrcA} ensures that in linear time we can also find a linear order $\leqR$ of the rows of $M$ such that $(\leqR,\leqC)$ is a monotone circular biorder. Once again, as we are assuming that $M$ has no trivial rows, Lemma~\ref{lem:Basu} implies $(\leqR,\leqC)$ is a circularly compatible ones biorder.

It only remains to show that the assumption that $M$ has no trivial rows is indeed without loss of generality. Suppose that $M$ has some trivial row all whose entries are equal to $u$ for some $u\in\{0,1\}$. Notice that $\size(M^{[u]})\in O(\size(M))$. Hence, we can apply the procedure described in the preceding paragraph to $M^{[u]}$ (instead of $M$) in order to produce in linear time either a circularly compatible ones biorder of $M^{[u]}$ or a matrix $F$ in $\FDCircR^\infty$ contained in $M^{[u]}$ as a configuration. On the one hand, a circularly compatible ones biorder of $M^{[u]}$ induces a circularly compatible ones biorder of $M$ by restriction to the rows and columns of $M$. On the other hand, given some matrix $F$ in $\FDCircR^\infty$ contained in $M^{[u]}$ as a configuration, the proof of Lemma~\ref{lem:M[u]} gives a linear-time procedure for finding a matrix $F'$ contained in $M$ as a configuration and such that $F'\in\FDCircR^\infty$ or $(F')\trans\in\FDCircR^\infty$, which leads to a matrix in $\FCCO^\infty$ contained in $M$ as a configuration in additional $O(1)$ time (see the proof of Lemma~\ref{lem:CCO=>Dcirc}). Hence, the whole algorithm can be completed in linear time. This completes the proof of the theorem.
\end{proof}

\section{Proper circular-arc bigraphs}\label{sec:PCAB+PIB}

The main results in this section are a characterization by minimal forbidden induced subgraphs and a linear-time recognition algorithm for proper circular-arc bigraphs. These results arise by combining the results obtained in the preceding sections with the following result (and its proof).

\begin{thm}[\cite{MR3065109}]\label{thm:BasuPCAb} Let $G$ be a bipartite graph and let $M$ be a biadjacency matrix of $G$. If $M$ has no trivial rows, then the following assertions are equivalent:
\begin{enumerate}[(i)]
\item\label{it:BasuPCAb1} $M$ has the $D$-circular property;

\item\label{it:BasuPCAb2} $G$ is a proper circular-arc bigraph.
\end{enumerate}\end{thm}

We give the details of the construction of the proper circular-arc bimodel in~\cite{MR3065109} for the proof of \eqref{it:BasuPCAb1}${}\Rightarrow{}$\eqref{it:BasuPCAb2} in the above theorem.
Let $G$ and $M$ be as in Theorem~\ref{thm:BasuPCAb} and suppose assertion~\eqref{it:BasuPCAb1} holds. By permuting the columns (if necessary), we assume, without loss of generality, that the canonical order of the columns of $M$ is a $D$-circular order of $M$. Let $k$ and $\ell$ be the number of rows and columns of $M$, respectively. As usual, we identify the rows and columns of $M$ with their corresponding row and column indices. Thus, for each $i\in[k]$, row $i$ of $M$ equals $[a_i,b_i]_{[\ell]}$ for some $a_i,b_i\in[\ell]$. Let $\mathcal C$ be a circle with $\ell$ different points on it, labeled with $1$, $2$, \ldots, and $\ell$ in clockwise order. If $a,b\in[\ell]$, let $[a,b]_{\mathcal C}$ be the closed arc of $\mathcal C$ which extends in the clockwise direction from $a$ to $b$. By construction, $G$ is the bipartite intersection graph of the families $\mathcal F_1=\{[a_i,b_i]_{\mathcal C}\colon\,i\in[k]\}$ and $\mathcal F_2=\{\{j\}\colon\,j\in[\ell]\}$ and if two arcs $A_1$ and $A_2$ of $\mathcal F_1$ are such that $A_1\subseteq A_2$, then $A_1$ and $A_2$ share at least one endpoint. Therefore, it is possible to slightly perturb the endpoints of the arcs in $\mathcal F_1$ so that $\{\mathcal F_1,\mathcal F_2\}$ becomes a proper circular-arc bimodel of $G$ and thus assertion~\eqref{it:BasuPCAb2} holds. In this way, if $M$ is given together with some $D$-circular order of it, then a proper circular-arc bimodel of $G$ can be found in linear time (i.e., in $O(n+m)$ time).

\begin{figure}[t!]
\ffigbox[\textwidth]{%
\ffigbox[\FBwidth]{%
\begin{subfloatrow}
\ffigbox[0.20\textwidth]{\includegraphics{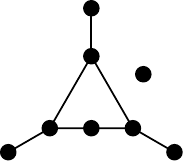}}{\caption{Bipartite graph associated with $\ZetaDosAst$ (and $\ZetaCuatroAst$)}}
\ffigbox[0.20\textwidth]{\includegraphics{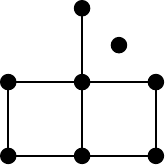}}{\caption{Bipartite graph associated with $\ZetaTresAst$ (and $\CoZetaCuatroAst$)}}\quad
\ffigbox[0.20\textwidth]{\includegraphics{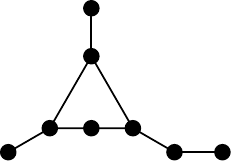}}{\caption{Bipartite graph associated with $\ZetaCinco$}}
\ffigbox[0.20\textwidth]{\includegraphics{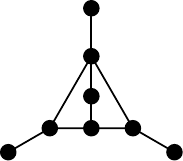}}{\caption{Bipartite graph associated with $\CoZetaDosAst$}}
\end{subfloatrow}}{}
}{\caption{Bipartite graphs in $\ForbPCAB$ associated with matrices in $\FCCO$}\label{fig:someFPCAB}}
\end{figure}

We now give our linear-time recognition algorithm for proper circular-arc bigraphs. When the input graph is not a proper circular-arc bigraph, the algorithm outputs some induced subgraph of the input graph that belongs to the following family of minimal forbidden induced subgraphs:
\[ \ForbPCAB=\{H\colon\,H\text{ is the bipartite graph associated with some matrix in $\FCCO^\infty$}\}.  \]
The graphs in $\ForbPCAB$ are those depicted in Figure~\ref{fig:someFPCAB} plus the bipartite graphs associated with $\MIast k$ and $\overline{\MIast k}$ (i.e., the chordless cycle on $2k$ vertices plus an isolated vertex, and its bipartite complement, respectively) for each $k\geq 3$. 

\begin{thm}\label{thm:algo_PCAB} There is a linear-time algorithm that, given any bipartite graph $G$, outputs either a proper circular-arc bimodel of $G$ or a graph in $\ForbPCAB$ contained in $G$ as an induced subgraph.\end{thm}
\begin{proof} Let $M$ be any biadjacency matrix of $G$. We apply the algorithm of Theorem~\ref{thm:algo_CCO} to $M$. If the output is some matrix in $\FCCO^\infty$ contained in $M$ as a configuration, this immediately leads to a graph in $\ForbPCAB$ contained in $G$ as an induced subgraph. Thus, we assume, without loss of generality, that $M$ has the circularly compatible ones property. If $M$ has no trivial rows, we can apply Theorem~\ref{thm:algo_Dcirc} to obtain a $D$-circular order of $M$ and apply the construction in the proof of \eqref{it:BasuPCAb1}${}\Rightarrow{}$\eqref{it:BasuPCAb2} of Theorem~\ref{thm:BasuPCAb} given above in order to produce a proper circular-arc bimodel of $G$ in linear time. Hence, we assume, without loss of generality, that $M$ has some trivial row all whose entries are equal to $u$ for some $u\in\{0,1\}$. By Theorem~\ref{thm:CCO} and Lemma~\ref{lem:M[u]}, $M^{[u]}$ has the $D$-circular property. Hence, as $M^{[u]}$ has no trivial rows, we can proceed as the proof of \eqref{it:BasuPCAb1}${}\Rightarrow{}$\eqref{it:BasuPCAb2} of Theorem~\ref{thm:BasuPCAb} given above in order to produce a proper circular-arc bimodel $\{\mathcal F_1,\mathcal F_2^\ast\}$ of the bipartite graph associated with $M^{[u]}$ in linear time, where the arc $A^*$ for the vertex of $G$ corresponding to the last column of $M^{[u]}$ belongs $\mathcal F_2^\ast$. By removing $A^*$ from $\mathcal F_2^\ast$, we obtain a proper circular-arc bimodel of $G$. As $\size(M)\in O(n+m)$ and also $\size(M^{[u]})\in O(n+m)$, the whole algorithm takes linear time, completing the proof of the theorem.\end{proof}

Our next result gives some characterizations for arbitrary proper circular-arc bigraphs, including a characterization by minimal forbidden induced subgraphs.

\begin{thm}\label{thm:PCAb} If $G$ is a bipartite graph and $M$ is a biadjacency matrix of $G$, then the following assertions are equivalent:
\begin{enumerate}[(i)]
 \item\label{it1:PCAb} $G$ is a proper circular-arc bigraph;

 \item\label{it2:PCAb} $G$ contains no graph in $\ForbPCAB$ as an induced subgraph;
 
 \item\label{it3:PCAb} $M$ has the doubly $D$-circular property;
 
 \item\label{it4:PCAb} $M$ has the circularly compatible ones property.
\end{enumerate}\end{thm}
\begin{proof} Assertion~\eqref{it1:PCAb} implies assertion~\eqref{it2:PCAb} because no graph in $\ForbPCAB$ is a proper circular-arc bigraph and every induced subgraph of a proper circular-arc bigraph is a proper circular-arc bigraph. Conversely, Theorem~\ref{thm:algo_PCAB} shows that assertion~\eqref{it2:PCAb} implies assertion \eqref{it1:PCAb}. The equivalence between assertions~\eqref{it2:PCAb}, \eqref{it3:PCAb}, and~\eqref{it3:PCAb} follows immediately from Theorem~\ref{thm:CCO}. The proof of the theorem is complete.
\end{proof}

One may proceed in a similar fashion for proper interval bigraphs. Given a bipartite graph $G$, one may apply Theorem~\ref{thm:algo_Dint} to any biadjacency matrix of $M$ in order to produce either a $D$-interval order of $M$ or some matrix $F$ in $\FDIntR^\infty$ contained in $M$ as a configuration. On the one hand, if the output is a $D$-interval model, then a proper interval bimodel of $G$ can be built as in the proof of \eqref{it:BasuPCAb1}${}\Rightarrow{}$\eqref{it:BasuPCAb2} of Theorem~\ref{thm:BasuPCAb} given above. On the other hand, if the output is some matrix $F$ in $\FDIntR^\infty$, this immediately leads to an induced subgraph of $G$ that belongs to \begin{align*}
   \ForbPIB&=\{H\colon\,H\text{ is the bipartite graph associated with some matrix in $\FDIntR^\infty$}\}\\
           &=\{\text{bipartite claw},\text{bipartite net},\text{bipartite tent}\}\cup\{C_{2k}\colon\;k\geq 3\}, 
\end{align*}
where the bipartite claw, the bipartite net, and the bipartite tent are depicted in Figure~\ref{fig:someFPIB} and $C_{2k}$ denotes the chordless cycle on $2k$ vertices. In this way, one obtains a new linear-time recognition algorithm for proper interval graphs which, like the algorithm devised by Hell and Huang~\cite{MR2134416}, is able to find a minimal forbidden induced subgraph in any bipartite graph that is not a proper interval bigraph.

\begin{figure}[t!]
\ffigbox[\textwidth]{%
\ffigbox[\FBwidth]{%
\begin{subfloatrow}
\ffigbox[0.20\textwidth]{\includegraphics{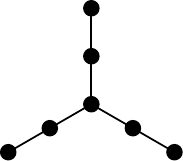}}{\caption{bipartite claw (associated with $Z_1$)}}
\ffigbox[0.20\textwidth]{\includegraphics{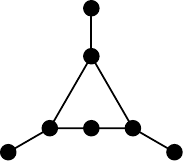}}{\caption{bipartite net (associated with $Z_2$)}}\quad
\ffigbox[0.20\textwidth]{\includegraphics{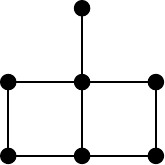}}{\caption{bipartite tent (associated with $Z_3$)}}
\end{subfloatrow}}{}
}{\caption{Bipartite claw, bipartite tent, and bipartite tent}\label{fig:someFPIB}}
\end{figure}

\begin{thm}[\cite{MR2134416}] There is a linear-time algorithm that, given any bipartite graph $G$, outputs either a proper interval bimodel of $G$ or a graph in $\ForbPIB$ contained in $G$ as an induced subgraph.\end{thm}

Indeed, as observed in~\cite{MR2071482}, since none of the graphs in $\ForbPIB$ is a proper interval bigraph, the above result implies that $\ForbPIB$ is the set of minimal forbidden induced subgraphs for the class of proper interval bigraphs (a fact that also follows, for instance, combining the works~\cite{MR0221974} and~\cite{MR1368750}).

\section*{Acknowledgments}

This work was partially supported by ANPCyT Grant PICT-2017-1315 and Universidad Nacional del Sur Grant PGI 24/L103.

%\bibliographystyle{abbrv}
%\bibliography{circularly}

\setcounter{section}{0}
\renewcommand\thesection{\Alph{section}}

\section{Appendix}\label{sec:app}

\subsection{Proof of Lemma~\ref{lem:R}}\label{app:lem:R}

\begin{proof}[Proof of Lemma~\ref{lem:R}] Let the \emph{conjugate} of a senary sequence $\lambda$ be the sequence that arises from the reversal of $\lambda$ by interchanging $4$'s with $5$'s. We begin by proving the following claim.

\begin{enumerate}[{Claim }1:]
\item \emph{If $\lambda=\lambda_1\lambda_2\ldots\lambda_k$ is a senary sequence of length $k$ for some $k\geq 3$, then $R(\overline\lambda)=\overline{R(\lambda)}$. Moreover, if $\mu$ arises from $\lambda$ by applying a sequence of shifts and conjugations, then $R(\mu)$ represents the same configuration as $R(\lambda)$.}
\begin{proof} The equality $R(\overline\lambda)=\overline{R(\lambda)}$ follows immediately from definition. In order to prove the second statement of the lemma, it suffices to prove that $R(\mu)$ represents the same configuration as $R(\lambda)$ when $\mu$ is the shift or the conjugate of $\lambda$ (as the general case would then follow by induction). Let $m$ be the number of rows of $R(\lambda)$. Suppose first that $\mu$ is the shift of $\lambda$. If we define $\rho$ as $\langle 2,3,\ldots,m,1\rangle$ or $\langle 3,4,\ldots,m,1,2\rangle$, depending on whether $\lambda_1\in\{0,1\}$ or not, respectively, and we define $\sigma=\langle 2,3,\ldots,k,1,k+1\rangle$, then $R(\mu)=R(\lambda)_{\rho,\sigma}$. Suppose now that $\mu$ is the conjugate of $\lambda$. For each $i\in[k]$, let $\tilde Q_{\lambda_i}(i,k)$ denote $Q_{\lambda_i}(i,k)$ whenever $\lambda_i\notin\{4,5\}$ and the matrix that arises by switching the two rows of $Q_{\lambda_i}(i,k)$ otherwise. If $\sigma=\langle 1,k,k-1,\ldots,2,k+1\rangle$, then $R(\mu)_{\id_m,\sigma}$ is formed precisely by the rows of $\tilde Q_{\lambda_k}(k,k)$, followed by the rows of $\tilde Q_{\lambda_{k-1}}(k-1,k)$, $\ldots$, followed by the rows of $\tilde Q_{\lambda_1}(1,k)$. Therefore, $R(\mu)_{\id_m,\sigma}$ arises from $R(\lambda)$ by rearranging the rows; in particular, $R(\mu)$ and $R(\lambda)$ represent the same configuration. This completes the proof of the lemma.\end{proof}
\end{enumerate}

Since the complement of each matrix in $\FDCircR^\infty$ represents the same configuration as some matrix in $\FDCircR^\infty$, it follows from the above claim that the lemma holds for some sequence $\lambda$ if and only if it holds for any sequence that arises from $\lambda$ by applying a sequence of shifts, complementations, and conjugations. 

Let $\lambda=\lambda_1\lambda_2\ldots\lambda_k$ be a senary sequence of length $k$ for some $k\geq 3$. Moreover, if $k=3$, we suppose further that neither $4$ nor $5$ occurs in $\lambda$. For each $i\in[k]$, we define $p_i$ as the index of the row of $R(\lambda)$ corresponding to the first row of $Q_{\lambda_i}(i,k)$. If $\lambda_i\notin\{0,1\}$, then we also define $q_i=p_i+1$ (i.e., the index of the row of $R(\lambda)$ corresponding to the second row of $Q_{\lambda_i}(i,k)$). We consider several cases. In each of the cases, it is always assumed that none of the preceding cases holds.

\begin{enumerate}[{Case }1:]
\item\label{case:01}\emph{$\lambda$ is a binary sequence.} By construction, $R(\lambda)=\lambda\miop\MIast k$. By applying shifts and conjugating $\lambda$ (if necessary), we assume, without loss of generality, that $\lambda$ is a bracelet. Moreover, if $k=3$, we may assume that $\lambda\in A_3$ because $011\miop\MIast 3$ represents the same configuration as $000\miop\MIast 3$ and $001\miop\MIast 3$ represents the same configuration as $111\miop\MIast 3$. Thus, $R(\lambda)\in\ForbRow$ and, by virtue of Lemma~\ref{lem:forb_Circ1R}, $R(\lambda)$ contains some matrix in $\FDCircR^\infty$ as a configuration.

\item\label{case:23}\emph{$23$, $24$, $32$, $35$, $43$, or $52$ occurs circularly in $\lambda$.} By shifting, complementing, and/or conjugating $\lambda$ (if necessary), we assume, without loss of generality, that $23$ or $24$ occurs in $\lambda$ at position $1$. On the one hand, if $23$ occurs in $\lambda$ at position $1$, then $R(\lambda)_{\langle 4,1,2,3\rangle, \langle k+1,2,3,1\rangle}=\ZetaCincoTrans$. On the other hand, if $24$ occurs in $\lambda$ at position $1$, then $R(\lambda)_{\langle 1,3,4,2\rangle,\langle 3,1,2,k+1\rangle}=\ZetaCincoTrans$.

\item\label{case:k3}\emph{$k=3$.} There only possible values of $\lambda$ up to shifts, reversals, and conjugations are $002$, $003$, $012$, $022$, $033$, or $222$. If $\lambda=002$, then $R(\lambda)_{\langle 4,1,3,2\rangle,\langle 2,1,3,4\rangle}=\ZetaTresAst$. If $\lambda=003$, then $R(\lambda)_{\langle 4,1,2,3\rangle,\langle 1,3,4,2\rangle}=\ZetaSiete$. If $\lambda=012$, then $R(\lambda)_{\langle 2, 1, 3, 4\rangle}=\ZetaCinco$. If $\lambda=022$, then $R(\lambda)_{\langle 3,1,2,5\rangle}=\ZetaDosAst$. If $\lambda=033$, then $R(\lambda)_{\langle 2,1,5,3\rangle,\langle 3,2,1,4\rangle}=\CoZetaSeis$. Finally, if $\lambda=222$, then $R(\lambda)_{\langle 1,4,6,2\rangle}=\ZetaUnoAst$.

\item\label{case:k>=4+(23)}\emph{Either $2$ or $3$ occurs in $\lambda$.} As Case~\ref{case:k3} does not hold, $k\geq 4$. By shifting and complementing $\lambda$ (if necessary), we suppose, without loss of generality, that $2$ occurs in $\lambda$ at position $1$. Since Case~\ref{case:23} does not hold, $\lambda_2\in\{0,1,2,5\}$. If $\lambda_2=0$, then $R(\lambda)_{\langle 2,3,1\rangle,\langle 2,1,k+1,4,3\rangle}=\CoZetaCuatroAst$. If $\lambda_2=2$, then $R(\lambda)_{\langle 2,4,1\rangle,\langle 1,2,k+1,3,4\rangle}=\CoZetaCuatroAst$. If $\lambda_2=5$, then $R(\lambda)_{\id_4,\langle k+1,2,4,3\rangle}=\CoZetaDosAst$. Thus, we assume, without loss of generality, that $\lambda_2=1$. By the same reasoning applied to the conjugate of $\lambda$, we assume, without loss of generality, that $\lambda_k=1$. Hence, $R(\lambda)_{\langle 2,3,p_k,1\rangle,\langle 2,1,k+1,3\rangle}=\CoZetaSeis$.

\item\label{case:45}\emph{$(\lambda_i,\lambda_{i+2})\in\{(4,4),(4,5),(5,4),(5,5)\}$ for some $i\in[k]$}. By shifting and complementing $\lambda$ (if necessary), we assume, without loss of generality that $\lambda_1=4$ and $\lambda_3\in\{4,5\}$. As Case~\ref{case:k3} does not hold, $k\geq 4$. On the one hand, if $\lambda_3=4$, then $R(\lambda)_{\langle 1,p_3,q_3,2\rangle,\langle 4,k+1,3,2,1\rangle}=\ZetaOcho$. On the other hand, if $\lambda_3=5$, then $R(\lambda)_{\langle q_3,1,2,p_3\rangle,\langle 2,k+1,1,3,4\rangle}=\ZetaOcho$.

\item\label{case:44}\emph{$44$ occurs circularly in $\lambda$.} By shifting $\lambda$ (if necessary), we assume without loss of generality that $44$ occurs in $\lambda$ at position $1$. Thus, $R(\lambda)_{\langle 3,1,2,4\rangle,\langle 2,k+1,1, 3))}=\ZetaCincoTrans$.

\item\label{case:(0*4)}\emph{$(\lambda_i,\lambda_{i+2})\in\{(0,4),(4,1),(1,5),(5,0)\}$ for some $i\in[k]$.} By shifting, complementing, and conjugating $\lambda$ (if necessary), we assume, without loss of generality, that $\lambda_1=0$ and $\lambda_3=4$. As Case~\ref{case:k3} does not hold, $k\geq 4$. As none of Cases~\ref{case:k>=4+(23)} and~\ref{case:44} holds, one of $004$, $014$, and $054$ occurs in $\lambda$ at position $1$ and, as a consequence, $R(\lambda)_{\id_4,\langle 2,3,4,k+1\rangle}$ is $\ZetaDosAst$, $\ZetaSeis$, or $\ZetaCincoTrans$, respectively.

\item\label{case:k=4}\emph{None of the preceding cases holds.} As none of Cases~\ref{case:01} and \ref{case:k>=4+(23)} holds, $4$ or $5$ occurs in $\lambda$. By shifting and complementing $\lambda$ (if necessary), we assume, without loss of generality, that $4$ occurs in $\lambda$ at position $1$. As none of the Cases~\ref{case:45} and~\ref{case:(0*4)} holds, $\lambda_{k-1}=1$ and $\lambda_3=0$. In particular, $k\neq 4$. Moreover, since Case~\ref{case:k3} does not hold, $k\geq 5$. Thus, $R(\lambda)_{\langle 1,p_{k-1},2,p_3\rangle,\langle k-1,k+1,1,2\rangle}=\ZetaCincoTrans$.
\end{enumerate}
As we have considered all the cases, the proof of the lemma is complete.\end{proof}

\subsection{Proof of Lemma~\ref{lem:W}}\label{app:lem:W}

\begin{proof}[Proof of Lemma~\ref{lem:W}] For each $i\in[4]$, we define $p_i$ as the index of the row of $W(\lambda)$ corresponding to the first row of $U_{\lambda_i}(i)$. If $\lambda_i\notin\{0,1\}$, then we also define $q_i=p_i+1$ (i.e., the index of the row of $W(\lambda)$ corresponding to the second row of $U_{\lambda_i}(i)$).

Since $W(\lambda)=\overline{W(\overline\lambda)}$ and the complement of each matrix in $\FDCircR$ represents the same configuration as some matrix in $\FDCircR$, it follows that the lemma holds for $\lambda$ if and only if it holds for $\overline\lambda$.

We consider several cases. In each case, it is assumed that none of the preceding cases holds.

\begin{enumerate}[{Case }1:]
 \item\label{caseW:01}\emph{$\lambda$ is binary.} Thus, $W(\lambda)=\lambda\miop\MIV$ which, by Theorem~\ref{thm:circR}, represents the same configuration as some matrix in $\ForbRow$. Hence, $W(\lambda)$ represents the same configuration as $\MIV$, $\overline\MIV$, $\MVast$, or $\overline{\MVast}$. By the proof of Lemma~\ref{lem:forb_Circ1R}, $W(\lambda)$ contains $\ZetaDosAst$, $\CoZetaDosAst$, $\ZetaSeis$, or $\CoZetaSeis$ as a configuration.

 \item\label{caseW:2or3}\emph{$2$ or $3$ occurs in $\lambda$}. Let $i\in[3]$ such that $\lambda_i\in\{2,3\}$. By complementing $\lambda$ (if necessary), we assume, without loss of generality, that $\lambda_i=2$. If we let $\sigma=\langle 2i+2,2i,2i-2,2i-1,2i+1\rangle$ or $\sigma=\langle 2i+1,2i-1,2i+3,2i,2i+2\rangle$ (where arithmetic is modulo $6$), depending on whether $\lambda_4$ is $0$ or $1$, respectively, then $W(\lambda)_{\langle p_4,p_i,q_i\rangle,\sigma}=\ZetaCuatroAst$. 
 
 \item\emph{$\lambda_1\neq\lambda_2$.} As Case~\ref{caseW:2or3} does not hold, by complementing $\lambda$ (if necessary), we assume, without loss of generality, that $\lambda_1=0$ and $\lambda_2=1$. If we let $\sigma=\langle 5,6,2,4,3\rangle$ or $\sigma=\langle 6,5,1,3,4\rangle$, depending on whether $\lambda_4$ is $0$ or $1$, respectively, then $W(\lambda)_{\langle 2,p_4,1\rangle,\sigma}=\ZetaCuatroAst$.
  
 \item\emph{None of the preceding cases holds.} In this case, $\lambda_1\in\{0,1\}$, $\lambda_2=\lambda_1$, and $\lambda_3\in\{4,5\}$. By complementing $\lambda$ (if necessary), we may assume, without loss of generality, that $\lambda_1=\lambda_2=0$. Thus, $W(\lambda)_{\id_4,\langle 2,4,5,6\rangle}=\ZetaUnoAst$.
\end{enumerate}
As we have exhausted all the cases, the proof of the lemma is complete.\end{proof}

\subsection{Proof of Lemma~\ref{lem:X}}\label{app:lem:X}

\begin{proof}[Proof of Lemma~\ref{lem:X}] If $\alpha=\alpha_1\alpha_2\alpha_3\alpha_4$ is a binary sequence of length $4$, we define the \emph{swap of $\alpha$}, denoted by $\alpha^\sharp$, as the sequence $\alpha_3\alpha_4\alpha_1\alpha_2$. We begin by proving the following claim.
\begin{enumerate}[{Claim}:]
 \item\emph{If $\alpha$ is a binary sequence of length $4$ and $\beta$ is any sequence that arises from $\alpha$ by applying a sequence of complementation and swap operations and $i\in[3]$, then $X_i(\beta)$ represents the same configuration as some matrix arising from $X_1(\alpha)$ by complementing some (eventually zero) rows.} We first notice that all three matrices $X_1(\beta)$, $X_2(\beta)$, and $X_3(\beta)$ represent the same configuration because it is straightforward to verify that, for each $i\in[3]$, $X_{i+1}(\beta)_{\langle 2,3,1,4,5,6\rangle,\langle 3,4,5,6,1,2\rangle}=X_i(\beta)$ (where $i+1$ is modulo $3$). Thus, we assume, without loss of generality, that $i=1$. The claim follows by induction from the following facts that hold for each binary sequence $\gamma$ of length $4$: $X_1(\overline{\gamma})_{\langle 1,2,3,4,6,5\rangle}=000011\miop X_1(\gamma)$ and $X_1(\gamma^\sharp)_{\langle 1,3,2,4,5,6\rangle,\langle 1,2,5,6,3,4\rangle}=X_1(\gamma)$. 
\end{enumerate}

Let $\alpha$ be a binary sequence of length $4$ such that $\alpha\notin\{0000,0011,1100,1111\}$. Because of the above claim, it suffices to prove the lemma for $i=1$. Moreover, as Theorem~\ref{thm:circR} ensures that any matrix that arises from $\ForbRow$ by complementing some of its rows represents the same configuration as some matrix in $\ForbRow$, the lemma holds for $\alpha$ if and only if it holds for $\overline\alpha$ or $\alpha^\sharp$. In particular, by complementing $\alpha$ (if necessary), we assume, without loss of generality, that $\alpha_1=0$. We consider a few cases.
\begin{enumerate}[{Case }1:]
 \item\label{caseX:1} \emph{$\alpha=0001$.} In this case, $X_1(\alpha)_{\langle 3,4,5\rangle,\langle 1,4,5,6\rangle}=\overline{\MIast 3}$.

 \item\label{caseX:2} \emph{$\alpha=0101$.} In this case, $X_1(\alpha)_{\langle 2,6,3,4\rangle, \langle 3,4,6,5,2\rangle}=0001\miop\MIast 4$.
 
 \item\label{caseX:3} \emph{$\alpha=0x10$ for some $x\in\{0,1\}$.} In this case, $X_1(\alpha)_{\langle 3,4,5\rangle,\langle 5,6,2,3\rangle}=\MIast 3$.

 \item\label{caseX:4} \emph{$\alpha=0100$.} This case reduces to Case~\ref{caseX:1} by replacing $\alpha$ by $\alpha^\sharp{}\mathbin{(=}0001)$.

 \item\label{caseX:5} \emph{$\alpha=011x$ for some $x\in\{0,1\}$.} This case reduces to Case~\ref{caseX:3} by replacing $\alpha$ by $\overline{\alpha^\sharp}{}\mathbin{(=}0\overline x10)$.
\end{enumerate}
Suppose, for a contradiction, that none of the preceding cases holds. We claim that $\alpha_4=1$. In fact, if $\alpha_4=0$, then, since neither Case~\ref{caseX:3} nor Case \ref{caseX:4} holds, necessarily $\alpha_3=0$ and $\alpha_2=0$; i.e., $\alpha=0000$, a contradiction. Moreover, we also claim that $\alpha_3=0$. In fact, if $\alpha_3=1$, then, as Case~\ref{caseX:5} does not hold, necessarily $\alpha_2=0$; i.e., $\alpha=0011$, a contradiction. Hence, as we are assuming $\alpha_1=0$, either Case~\ref{caseX:1} or Case~\ref{caseX:2} holds, a contradiction. This contradiction proves that the above analysis covers all the possible cases and so the proof of the lemma is complete.\qedhere

\end{proof}

\subsection{Proof of Lemma~\ref{lem:Y}}\label{app:lem:Y}

\begin{proof}[Proof of Lemma~\ref{lem:Y}] Clearly, if $\gamma'$ is the shift of $\gamma$, then $Y(\gamma')_{\langle 3,1,2,4,5,6\rangle,\langle 5,6,1,2,3,4\rangle}=Y(\gamma)$. Thus, the lemma holds for $\gamma$ if and only if it holds for $\gamma'$. It is also clear that $Y(\overline\gamma)_{\langle 1,2,3,4,6,5\rangle}=000011\miop Y(\gamma)$. Hence, as Theorem~\ref{thm:circR} ensures that any matrix that arises from a matrix of $\ForbRow$ by complementing some rows is equivalent to some matrix in $\ForbRow$, the lemma holds for $\gamma$ if and only if it holds for $\overline\gamma$. Therefore, as $\gamma$ is nonconstant, by applying shifts and complementations, we assume, without loss of generality, that $\gamma=001$ and, in this case, $Y(\gamma)_{\langle 1,4,2,5\rangle,\langle 1,2,4,3,5\rangle}=0001\miop\MIast 4$.\end{proof}

\end{document}